\documentclass[10pt]{amsart}

\usepackage{amssymb}
\usepackage{tikz-cd}
\usepackage[pagewise]{lineno}
\usepackage[bookmarksnumbered, bookmarksopen, 
colorlinks,citecolor=blue,linkcolor=blue,backref]{hyperref}

\newtheorem{theorem}{Theorem}[section]
\newtheorem{proposition}[theorem]{Proposition}
\newtheorem{lemma}[theorem]{Lemma}
\newtheorem{corollary}[theorem]{Corollary}

\theoremstyle{definition}
\newtheorem{definition}[theorem]{Definition}
\newtheorem{remark}[theorem]{Remark}

\newcommand{\qqed}{\hfill\Box}

\begin{document}
	
	
	\title[Semifree Cochain Models for Fibrations Revisited]{Semifree Cochain Models for Fibrations Revisited} 
	
	\author{Yanning Guo} 
	\address{Institute of Mathematics, Academy of Mathematics and Systems Science, Chinese Academy of Sciences, Beijing 100190, China} 
	\email{guoyanning@amss.ac.cn} 
	\urladdr{} 
	
	 \author{Ruizhi Huang} 
\address{State Key Laboratory of Mathematical Sciences \& Institute of Mathematics, Academy of Mathematics and Systems Science, Chinese Academy of Sciences, Beijing 100190, China
\newline
\indent School of Mathematical Sciences, University of Chinese Academy of Sciences, Beijing 100049, China} 
\email{huangrz@amss.ac.cn} 
   \urladdr{https://sites.google.com/site/hrzsea/}

	
	
	\begin{abstract}
We construct cochain models for Serre fibrations over path-connected bases
with coefficients in an arbitrary field. Our construction refines Brown's
twisted tensor product to produce a differential graded module model over the
singular cochain algebra of the base. If
the homology of the fibre is finite dimensional in each degree and on which the action
of the fundamental group of the base is degreewise nilpotent, then this model is a semifree
resolution generated by the cohomology of the fibre. This extends the
classical construction for simply connected bases. As applications, we
compare the resulting model with arbitrary semifree resolutions, showing that
every such resolution recovers the cohomology of the fibre after base change
to the basepoint, and obtain an explicit bound on the page at which an
associated spectral sequence collapses.
\end{abstract}

\maketitle

\tableofcontents

\section{Introduction}\label{intro}
Let $p\colon E\to B$ be a Serre fibration over a path-connected based space $(B,b_0)$, and let $F=p^{-1}(b_0)$ denote its fibre. A basic problem in algebraic topology is to construct an algebraic model of the Serre fibration
relative to a chosen algebraic model of the base, such as the singular cochain algebra $(C^*(B;\Bbbk),\delta_B)$, where $\Bbbk$ is the ground field.

Brown's twisted tensor products~\cite{B} provide a chain-level approach to modeling fibrations in terms of the base, the fibre, and a twisting cochain. Higher homotopy structures were used to develop this approach further. Kadeishvili~\cite{Kad80} extended twisted tensor products to $A_\infty$-algebras and $A_\infty$-modules and applied them to the homology of associated fibre bundles. Smirnov~\cite{Smi80} used strongly homotopy associative structures to transfer twisted tensor products to
resolutions of homology. Berikashvili~\cite{Ber06} constructed bigraded chain models for general Serre fibrations using local systems over the singular complex of the base.

Over a field of characteristic zero, Sullivan~\cite{Sul77} and
Quillen~\cite{Qui69} developed complementary algebraic approaches that gave
rise to the rich theory of rational homotopy. Within Sullivan's framework,
the semifree resolution is a fundamental tool for modeling fibrations. In fact, a classical semifree cochain model is available when the base is simply connected and suitable finite-type hypotheses are satisfied, where the ground field could be of arbitrary characteristic; see~\cite[Theorem~7.5]{FHT}. In the setting of simply connected spaces of finite type, Dupont and Hess~\cite{DH94} constructed twisted tensor models for fibrations over arbitrary fields using twisted extensions of cochain algebras.

For non-simply-connected bases, the fundamental group action on fibre homology plays an essential role.
Dwyer~\cite{Dwyer} identified degreewise nilpotence of this action as the criterion for strong convergence of the homological Eilenberg--Moore spectral
sequence with the chosen coefficients.

In this paper, we refine Brown's construction to obtain explicit
cochain models in the category of differential graded modules.
The construction uses a strict comodule refinement of Brown's
inverse and degreewise dualization. When the homology of the
fibre is of finite type and the fundamental group of the base
acts degreewise nilpotently on it, the construction yields a
semifree resolution of $C^*(E;\Bbbk)$ over the ordinary singular
cochain algebra $C^*(B;\Bbbk)$, generated by $H^*(F;\Bbbk)$.
The nilpotence filtration of fundamental group action on fibre homology enters explicitly into the construction of the semifree filtration. 

Our construction establishes a connection between semifree modules over singular cochains and Brown's twisted tensor products. The twisted tensor products and the related Hirsch--Brown model \cite{Hir, B} have
appeared in the work of Allday and Puppe~\cite{AP} and of
Amann and Zoller~\cite{AZ}, as well as in VandeBogert's recent
breakthrough~\cite{Van26} on the Halperin--Carlsson
Conjecture~\cite{Car83, Hal85} in characteristic two.

We recall the notion of a semifree module. Let $\Bbbk$ be a field of arbitrary
characteristic, and let $(R,d)$ be a differential graded algebra over
$\Bbbk$. An $(R,d)$--module $(M,d)$ is called
\emph{semifree over $(R,d)$}, or simply \emph{semifree} when the dga is
understood, if it admits an increasing filtration
\[
0=M(-1)\subset M(0)\subset M(1)\subset\cdots\subset M(k)
\subset M(k+1)\subset\cdots
\]
such that
\[
M=\bigcup_{k\geq -1}M(k),
\]
each $M(k)$ is a $(R,d)$--submodule of $M$, and, for every
$k\geq 0$, the quotient
\[
M(k)/M(k-1)
\]
is a free graded $R$--module generated by cycles. A \emph{semifree
resolution} of a $(R,d)$--module $(Q,d)$ is a
quasi--isomorphism
\[
m\colon (M,d)\xrightarrow{\simeq}(Q,d)
\]
of $(R,d)$--modules, where $(M,d)$ is semifree.

We say that $H_*(F;\Bbbk)$ is \emph{of finite type} if
$H_q(F;\Bbbk)$ is finite-dimensional for every $q$. The action of
$\pi_1(B,b_0)$ on $H_*(F;\Bbbk)$ is said to be \emph{degreewise nilpotent} if, for each $s\geq 0$, $H_s(F;\Bbbk)$ admits a finite filtration of length $N_s$ by subspaces that is
invariant under the action and such that the induced action on each
successive quotient is trivial. The action is further said to be \emph{nilpotent} if the sequence $\{N_s\}_{s\geq 0}$ admits an upper bound $N$; see
Definition~\ref{defnilact}.

For a chain complex $C$, we write $C^\vee$ for its degreewise linear dual.
We use the dual cochain convention of Definition~\ref{defdual}; under this
convention, the degreewise dual of the singular chain complex is naturally
isomorphic to the usual singular cochain complex. The
\emph{vertex--relative chain complex} $C_*(B|b_0;\Bbbk)$ is generated by
those singular simplices of $B$ all of whose vertices map to $b_0$. We set
\[
C^*(B|b_0;\Bbbk)=C_*(B|b_0;\Bbbk)^\vee.
\]
The chain coalgebra comparison of Proposition~\ref{propqiso} induces a dga
quasi--isomorphism
\[
h_B^\vee\colon
C^*(B|b_0;\Bbbk)
\xrightarrow{\simeq}
C^*(B;\Bbbk).
\]
The map
\[
p^\sharp\colon
C^*(B;\Bbbk)
\longrightarrow
C^*(E;\Bbbk)
\]
equips $C^*(E;\Bbbk)$ with its usual 
$(C^*(B;\Bbbk),\delta_B)$--module structure. Restriction of scalars along $h_B^\vee$
then equips it with the
$(C^*(B|b_0;\Bbbk),\delta_B)$--module structure used below.

Let $\mathcal E(F)$ denote the graded algebra of endomorphisms of
$H_*(F;\Bbbk)$ of nonnegative degree. A \emph{twisting cochain} is a degree $-1$ map
\[
\psi\colon
C_*(B|b_0;\Bbbk)
\longrightarrow
\mathcal E(F)
\]
in the sense of Definition~\ref{defcalE}.
As in
Definition~\ref{deftdiff}, such a twisting cochain determines a twisted
differential $\partial_\psi$ on
\[
C_*(B|b_0;\Bbbk)\otimes H_*(F;\Bbbk).
\]
We denote by $\delta_\psi$ the differential on
\[
\big(
C_*(B|b_0;\Bbbk)\otimes H_*(F;\Bbbk)
\big)^\vee
\]
obtained by degreewise dualizing $\partial_\psi$ according to the convention
of Definition~\ref{defdual}, and call it the \emph{dualized twisted
differential}. See also the discussion following Theorem~\ref{mthm1}.

\begin{theorem}\label{mthm1}
Let $p\colon E\to B$ be a Serre fibration over a path-connected based space
$(B,b_0)$, with fibre $F=p^{-1}(b_0)$.
\begin{itemize}
\item[(i)] There exist a twisting cochain
\[
\psi\colon
C_*(B|b_0;\Bbbk)
\longrightarrow
\mathcal E(F)
\]
and a quasi--isomorphism of
$(C^*(B|b_0;\Bbbk),\delta_B)$--modules
\[
\zeta\colon
\Big(
\big(
C_*(B|b_0;\Bbbk)\otimes H_*(F;\Bbbk)
\big)^\vee,
\delta_\psi
\Big)
\xrightarrow{\simeq}
C^*(E;\Bbbk).
\]
If $H_*(F;\Bbbk)$ is of finite type, the source of $\zeta$ is canonically
identified, as a $(C^*(B|b_0;\Bbbk),\delta_B)$--module, with
\[
\big(
C^*(B|b_0;\Bbbk)\otimes H^*(F;\Bbbk),
\delta_\psi
\big).
\]

\item[(ii)] Suppose that $H_*(F;\Bbbk)$ is of finite type and that the action
of $\pi_1(B,b_0)$ on $H_*(F;\Bbbk)$ is degreewise nilpotent. Then, under the
identification in~{\rm(i)}, the source of $\zeta$ is a semifree
$C^*(B|b_0;\Bbbk)$--module. Extension of scalars along $h_B^\vee$ yields a
semifree resolution
\[
m\colon
\big(
C^*(B;\Bbbk)\otimes H^*(F;\Bbbk),
\delta_\psi
\big)
\xrightarrow{\simeq}
C^*(E;\Bbbk)
\]
of $C^*(E;\Bbbk)$ as a 
$(C^*(B;\Bbbk),\delta_B)$--module.
\end{itemize}
\end{theorem}

The proof of Theorem~\ref{mthm1} begins with Brown's twisted tensor product
construction~\cite{B}. For what Brown calls a \emph{weakly transitive
fibration}, he constructs a twisted tensor product on
\[
C_*(B|b_0;\Bbbk)\otimes C_*(F;\Bbbk)
\]
together with a chain homotopy equivalence to the vertex--relative chain
complex $C_*(E|F;\Bbbk)$. His acyclic-model argument also produces a chain
homotopy inverse. The main technical step in our proof is to refine Brown's
construction of this inverse so that it becomes a strict morphism of
differential graded comodules over $C_*(B|b_0;\Bbbk)$. Brown's subsequent
replacement of the fibre chain complex by its homology produces the twisting
cochain $\psi$ appearing in the theorem. After degreewise dualization, the
refined inverse becomes a strict morphism of
$(C^*(B|b_0;\Bbbk),\delta_B)$--modules, rather than merely a morphism of the underlying
cochain complexes. Composing this morphism with the vertex--relative singular
cochain comparison
\[
C^*(E|F;\Bbbk)\xrightarrow{\simeq}C^*(E;\Bbbk)
\]
yields the quasi--isomorphism $\zeta$ in part~{\rm(i)}. Thus, the key
structural point of Theorem~\ref{mthm1} is the promotion of Brown's
chain-level construction to the category of differential graded modules over
the cochains of the base.

Since Brown's construction is formulated for weakly transitive fibrations, it
does not apply directly to an arbitrary Serre fibration. We first replace the
given fibration by its mapping-path fibration, which is a Hurewicz fibration
over the same base, and then apply Brown's transitive replacement. The
resulting comparisons induce a homotopy equivalence on total spaces and a
homology isomorphism on fibres, allowing the differential graded module model
to be transferred back to the original fibration. This procedure is carried
out in Subsections~\ref{ssechurrep} and~\ref{ssecttransrep} and culminates in
Theorem~\ref{theorembasedmodel}. The dualized twisted differential
$\delta_\psi$ is constructed as part of this procedure, and its explicit
formula under the finite-type identification in part~{\rm(i)} is given in
Lemma~\ref{lemdelta}.

The finite-type and degreewise nilpotence hypotheses serve distinct purposes. Neither is
needed for the degreewise dual model or for the quasi--isomorphism $\zeta$ in
part~{\rm(i)}. The finite-type hypothesis is used only to identify the source
of $\zeta$ with
\[
\big(
C^*(B|b_0;\Bbbk)\otimes H^*(F;\Bbbk),
\delta_\psi
\big),
\]
whereas degreewise nilpotence of the action of $\pi_1(B,b_0)$ on $H_*(F;\Bbbk)$ is
equivalent to the corresponding nilpotence condition on $\psi$. The latter
condition provides the filtration with respect to which the model is
semifree. Extension of scalars along $h_B^\vee$ then yields the resolution
$m$ in part~{\rm(ii)}. Degreewise nilpotence condition also appear in
Dwyer's convergence theorem for the Eilenberg--Moore spectral
sequence~\cite{Dwyer}; here, however, degreewise nilpotence arises directly at the chain
level by providing the semifree filtration on the dualized Brown model.

\begin{definition}
Under the hypotheses of Theorem~\ref{mthm1}{\rm(ii)}, the semifree resolution
\[
m\colon
\big(
C^*(B;\Bbbk)\otimes H^*(F;\Bbbk),
\delta_\psi
\big)
\xrightarrow{\simeq}
C^*(E;\Bbbk)
\]
constructed in Theorem~\ref{mthm1} is called a \emph{semifree cochain model}
of the Serre fibration $p$.
\end{definition}

The classical uniqueness theorem for semifree resolutions shows that the
semifree cochain model is homotopy equivalent, as a resolution of the $(C^*(B;\Bbbk),\delta_B)$-module
$C^*(E;\Bbbk)$, to any other semifree resolution; see
Proposition~\ref{proplift}. The following theorem combines this comparison
with a base-change statement: after reduction to the basepoint, every
semifree resolution recovers the cohomology of the fibre up to
quasi--isomorphism. It extends the corresponding result for simply connected
bases in~\cite[Theorem~7.5]{FHT} to the setting of
Theorem~\ref{mthm1}{\rm(ii)}.

Let
\[
\varepsilon_B\colon
C^*(B;\Bbbk)
\longrightarrow
\Bbbk
\]
be the augmentation induced by the basepoint $b_0$. For a
$(C^*(B;\Bbbk),\delta_B)$--module of the form
\[
\big(
C^*(B;\Bbbk)\otimes W,
d
\big),
\]
we write $(W,\bar d)$ for the differential graded vector space obtained by
extension of scalars along $\varepsilon_B$.

\begin{theorem}\label{mthm2}
Let $p\colon E\to B$ be a Serre fibration over a path-connected based space
$(B,b_0)$, with fibre $F=p^{-1}(b_0)$. Suppose that $H_*(F;\Bbbk)$ is of
finite type and that the action of $\pi_1(B,b_0)$ on $H_*(F;\Bbbk)$ is
degreewise nilpotent. Let
\[
m_W\colon
\big(C^*(B;\Bbbk)\otimes W,d\big)
\xrightarrow{\simeq}
C^*(E;\Bbbk)
\]
be a semifree resolution of $C^*(E;\Bbbk)$ as a
$(C^*(B;\Bbbk),\delta_B)$--module. Then there is a homotopy equivalence of $(C^*(B;\Bbbk),\delta_B)$--modules
\[
n\colon
\big(C^*(B;\Bbbk)\otimes W,d\big)
\xrightarrow{\simeq}
\big(C^*(B;\Bbbk)\otimes H^*(F;\Bbbk),\delta_\psi\big)
\]
such that $m\circ n$ and $m_W$ are homotopic as morphisms of $(C^*(B;\Bbbk),\delta_B)$--modules. Moreover, after extension of scalars along
$\varepsilon_B$, the map $n$ induces a quasi--isomorphism
\[
\bar n\colon
(W,\bar d)
\xrightarrow{\simeq}
\big(H^*(F;\Bbbk),0\big).
\]
\end{theorem}
In particular, Theorem~\ref{mthm2} identifies the differential graded Tor
group
\[
\operatorname{Tor}^{C^*(B;\Bbbk)}
\bigl(\Bbbk,C^*(E;\Bbbk)\bigr)
\cong
H^*(F;\Bbbk),
\]
where $\Bbbk$ is regarded as a right $C^*(B;\Bbbk)$--module via
$\varepsilon_B$, and Tor is computed using a semifree resolution. This
identifies the algebraic target of the Eilenberg--Moore spectral sequence.
The role of degreewise nilpotent actions in its convergence is classical.
Dwyer's theorem~\cite{Dwyer} gives a criterion for strong convergence in
terms of degreewise nilpotence of the action on the homology of the fibre,
with coefficients in an arbitrary abelian group. Section~\ref{secss}
explains the relation between Dwyer's convergence result and the cochain
model constructed here.

A semifree filtration gives rise to another spectral sequence, with the
homological indexing conventions specified in Section~\ref{secss}. When the
cohomology of the fibre is bounded, our construction yields a finite
filtration. The resulting spectral sequence converges to the cohomology of
the total space and admits an explicit collapse bound.

Suppose further that $\pi_1(B,b_0)$ acts nilpotently on $H_*(F;\Bbbk)$. The
\emph{nilpotence index}
\[
\mathrm{nil}\bigl(\pi_1(B,b_0),H_*(F;\Bbbk)\bigr)
\]
of the action is the least positive integer equal to the length of a
filtration on whose successive quotients the induced action is trivial; see
Definition~\ref{defnilact}.

\begin{theorem}\label{mthm3}
Let $p\colon E\to B$ be a Serre fibration over a path-connected based space
$(B,b_0)$, with fibre $F=p^{-1}(b_0)$. Suppose that
$H_*(F;\Bbbk)$ is of finite type, that the action of $\pi_1(B,b_0)$ on
$H_*(F;\Bbbk)$ is nilpotent, and that
\[
H^{\geq D}(F;\Bbbk)=0
\]
for some positive integer $D$. Set
\[
N=
\mathrm{nil}\bigl(\pi_1(B,b_0),H_*(F;\Bbbk)\bigr).
\]
Then the semifree cochain model admits a semifree filtration whose associated
spectral sequence, indexed homologically, collapses at the $E^{ND}$--page.
\end{theorem}

The bound arises from a filtration having at most $N$ stages in each fibre
cohomological degree. It therefore depends only on the nilpotence index of
the action and the range in which the cohomology of the fibre is nonzero.

$\, $

The paper is organized as follows.

Section~\ref{secsf} reviews the basic properties of semifree modules, including base change and comparison of semifree resolutions.

Section~\ref{secbr} recalls Brown's theory of twisted tensor products and his topological model for weakly transitive fibrations. We then recall the standard mapping-path Hurewicz replacement of a Serre fibration and Brown's transitive replacement. Together these constructions allow Brown's model to be applied to the Serre fibrations considered in the main theorems.

Section~\ref{secdttp} studies degreewise duals of twisted tensor products. We establish their differential graded module structure, analyze Brown's replacement of the fibre complex by its homology, identify the degreewise dual with $K^\vee\otimes H^*(L)$ under a finite-type hypothesis, derive an explicit formula for the dualized twisted differential, and give a $K_1$--nilpotence criterion for semifreeness.

Section~\ref{secsfm} contains the main topological construction. We first develop a vertex--relative singular-chain coalgebra comparison. We then refine Brown's inverse chain map to a strict comodule morphism, dualize this refinement to obtain a strict based cochain-module model, and use the Hurewicz and transitive replacements to treat general Serre fibrations. We identify topological degreewise nilpotence of the action with the algebraic $K_1$--nilpotence condition, prove semifreeness, and finally use base change to obtain the model over $C^*(B;\Bbbk)$ in Theorem~\ref{mthm1}. The comparison results of Theorem~\ref{mthm2} are proved at the end of the section.

Section~\ref{secss} studies spectral sequences associated with these semifree models. We prove the collapse estimate in Theorem~\ref{mthm3} and then place the resulting models in the context of the Eilenberg--Moore spectral sequence and Dwyer's convergence theorem.

Section~\ref{appendix} gives the detailed geometric construction of homotopy equivalences between vertex--relative and ordinary singular chain complexes with certain compatibility with Alexander--Whitney coproduct, including the face-compatible maps $CL^n$ and $\mathcal K^n$ and their compatibility with the projection of a Serre fibration.

$\, $

\noindent{\textbf{Acknowledgements.}} Ruizhi Huang was supported in part by the National Natural Science Foundation of China (Grant nos. 12331003 and 12288201), the National Key R\&D Program of China (No. 2021YFA1002300).

\section{Preliminaries on Semifree Modules}\label{secsf}
    This section recalls material from Chapter~6 of~\cite{FHT} for convenient reference. The detailed definition of differential graded algebras (dga's) and differential graded modules over a dga can also be found in \cite{FHT}.
	
	The semifree property of an $(R,d)$--module is preserved under base change:
	
	\begin{lemma}\label{lembc}
	Let $(M,d)$ be a semifree $(R,d)$--module and $(R,d)\rightarrow (S,d)$ be a $\Bbbk$--dga morphism, then $(S\otimes_R M,d)$ is a semifree $(S,d)$--module. $\qqed$
	\end{lemma}

\begin{lemma}\label{lemqibc}
Let
\[
\alpha:(R,d)\longrightarrow(S,d)
\]
be a quasi--isomorphism of $\Bbbk$--dgas, and let $(M,d)$ be a semifree left differential graded $(R,d)$--module. Then the canonical $\alpha$--linear morphism
\[
\beta_M:(M,d)\longrightarrow\big(S\otimes_RM,d\big),
\qquad
m\longmapsto1\otimes_Rm,
\]
is a quasi--isomorphism. $\qqed$
\end{lemma}

\noindent This is the standard base-change property for semifree modules; see  \cite[Chapter~6]{FHT}.

	A semifree filtration has free subquotients generated by cycles. The following lemma gives a useful weakening of this condition.
	
	\begin{lemma}\label{lemsemifree}
	If an $(R, d)$--module $(M, d)$ admits a filtration of $M$
	\[
		0=M(-1)\subset M(0)\subset M(1)\subset \cdots \subset M(k)\subset M(k+1)\subset \cdots
	\]
	\[
		M=\bigcup_{k\geq -1}M(k)
	\]
	such that $M(k)$ is a $(R,d)$--submodule of $(M,d)$ for each $k\geq 0$ and each subquotient $M(k)/M(k-1)$ is $(R,d)$--semifree, then $(M, d)$ is $(R,d)$--semifree itself. $\qqed$ 
	\end{lemma}
	
	We recall the following notion of homotopy between $(R,d)$--module morphisms.
	\begin{definition}\label{defRDModht}
		Let $f,g:(M,d_M)\rightarrow (N,d_N)$ be two morphisms of $(R,d)$--modules. We say $f$ and $g$ are {\it homotopic as $(R,d)$-module morphisms} if there exists an $R$-linear morphism $h:M\rightarrow N$ of degree $-1$ such that $f-g=d_N h+hd_M$, denoted by $f\simeq_R g$. 
	\end{definition}
	
	For a general $(R,d)$--module, one may replace it by a semifree resolution as in Section~\ref{intro}. The following lifting property gives the uniqueness of lifts from semifree modules and, together with the standard semifree replacement construction, the usual uniqueness statement for semifree resolutions:
	
	\begin{proposition}\label{proplift}
Let $\theta:(M,d)\to(Q,d)$ and $\eta:(P,d)\to(Q,d)$ be morphisms of $(R,d)$--modules. If $(M,d)$ is $(R,d)$--semifree and $\eta$ is a quasi--isomorphism, then there exists an $(R,d)$--module morphism
\[
\varphi:(M,d)\longrightarrow(P,d)
\]
such that $\eta\varphi\simeq_R\theta$. Its homotopy class is unique. In particular, any two semifree resolutions of the same $(R,d)$--module are related by a homotopy equivalence of $(R,d)$--modules compatible with the resolution maps up to $(R,d)$--module homotopy. $\qqed$
\end{proposition}

\section{Brown's Twisted Tensor Product Model for a Fibration}\label{secbr}
This section recalls the algebraic foundations of Brown's twisted tensor products, Brown's topological model for weakly transitive fibrations, and the replacement constructions needed to apply that model to Serre fibrations.
	\subsection{Twisted Tensor Products over a Commutative Ring}\label{ssecttp}
	In this subsection, we fix the ground ring to be a commutative ring $\Lambda$ with unit $1_\Lambda$. All differentials in this subsection are of degree $-1$ and denoted by $\partial$.
	\begin{definition}\label{defdgc}
We use the following terminology. All complexes in this subsection are nonnegatively graded and have differential of degree $-1$.
\begin{itemize}
\item[(i)] An \emph{augmented $\Lambda$--complex} is a chain complex $(M,\partial_M)$ together with a degree-zero chain map
\[
\varepsilon_M:M\longrightarrow\Lambda,
\]
where $\Lambda$ is concentrated in degree $0$, such that $\varepsilon_M$ is surjective. The map $\varepsilon_M$ is called the \emph{augmentation}.

\item[(ii)] A \emph{counital differential graded $\Lambda$--coalgebra}, or $\Lambda$--dgc, is a chain complex $(K,\partial_K)$ equipped with a coassociative chain map $\Delta_K:K\to K\otimes K$ and a surjective chain map $\varepsilon_K:K\to\Lambda$, called the \emph{counit}, satisfying
\[
(\varepsilon_K\otimes\mathrm{id}_K)\Delta_K
=\mathrm{id}_K
=(\mathrm{id}_K\otimes\varepsilon_K)\Delta_K.
\]
Following Brown, we require the counit to be surjective, so the same map is an augmentation of the underlying chain complex in the sense of~(i). We reserve the word ``counit'' for its coalgebraic role and ``augmentation'' for the underlying complex or algebra. A \emph{coaugmentation} is a morphism of dg coalgebras $\eta_K:\Lambda\to K$ satisfying $\varepsilon_K\eta_K=\mathrm{id}_\Lambda$. We write $1_K:=\eta_K(1)$ and $\bar K:=\ker\varepsilon_K$. Thus $\partial_K1_K=0$, $\Delta_K(1_K)=1_K\otimes1_K$, and $\varepsilon_K(1_K)=1$; moreover, for $k\in\bar K$,
\[
\bar\Delta_K(k):=\Delta_K(k)-1_K\otimes k-k\otimes1_K\in\bar K\otimes\bar K.
\]

\item[(iii)] An \emph{augmented differential graded $\Lambda$--algebra with degree $-1$ differential}, abbreviated $\Lambda$--$\partial$ga, is a unital differential graded algebra $(A,\partial_A,\mu_A,1_A)$ together with a morphism of unital dg algebras $\varepsilon_A:A\to\Lambda$. The map $\varepsilon_A$ is called the \emph{augmentation}; equivalently, it is a chain map satisfying $\varepsilon_A(ab)=\varepsilon_A(a)\varepsilon_A(b)$ and $\varepsilon_A(1_A)=1_\Lambda$.
\end{itemize}
The notation $\partial$ga is used here only to emphasize Brown's homological convention in which the differential has degree $-1$; later, ``dga'' refers to the cohomologically graded algebras obtained by dualization. We keep the standard abbreviation ``dgc'' for coalgebras, since no competing cohomological dgc convention is used in the paper.
\end{definition}

	Let $(K,\partial_K,\Delta_K, \varepsilon_K)$ be a $\Lambda$--dgc and let $(A,\partial_A,\mu_A, \varepsilon_A)$ be a $\Lambda$--$\partial$ga. The following definitions were introduced by Brown in \cite{B}:
	
	\begin{definition}\label{deftch}
	A graded $\Lambda$--linear morphism of degree $-1$
	\[
		\phi=\sum_{q\geq 0}\phi_q\in \bigoplus_{q\geq 0}\mathrm{Hom}_{\Lambda}(K_q,A_{q-1})=:\mathrm{Hom}^{-1}_{\Lambda}(K,A)
	\]
	is called a {\it twisting cochain} if it satisfies $\varepsilon_A\phi_1=0$ and the following relation:
	\begin{equation}\label{tweq}
		\ \partial_{A}\phi_q=\phi_{q-1}\partial_K-\sum_{k=1}^{q-1} (-1)^k\phi_{k}\cup \phi_{q-k}
	\end{equation}
	Here $\cup$ denotes Brown's cup product: for homogeneous cochains $u,v$ and $k\in K$,
	\[
	(u\cup v)(k):=\sum_{(k)}u(k_{(1)})v(k_{(2)}),
	\]
	with the same Brown sign convention used below for the cap product.
	\end{definition}
	
	Now let $(L,\partial_L, \varepsilon_L)$ be an augmented $\Lambda$--complex with an $A$--module structure defined by a morphism $\mu:A\otimes L\rightarrow L$ of augmented $\Lambda$--complexes, and $\phi$ be a twisting cochain in $\mathrm{Hom}^{-1}_{\Lambda}(K,A)$.
	
	\begin{definition}\label{deftdiff}
Let $\phi$ be a twisting cochain. The \emph{cap product induced by $\phi$}
is the degree $-1$ map
\[
-\cap\phi:K\otimes L\longrightarrow K\otimes L
\]
obtained from the coproduct of $K$, the map $\phi$, and the $A$--action on
$L$. Throughout this paper we use Brown's elementwise sign convention:
for homogeneous $k\in K$ and $l\in L$,
\begin{equation}\label{eqcapconvention}
\Delta_K(k)=\sum_{(k)}k_{(1)}\otimes k_{(2)},
\qquad
(k\otimes l)\cap\phi
:=\sum_{(k)} k_{(1)}\otimes \phi(k_{(2)})l.
\end{equation}
Thus the displayed composite
\[
K\otimes L\xrightarrow{\Delta_K\otimes\mathrm{id}}
K\otimes K\otimes L
\xrightarrow{\mathrm{id}\otimes\phi\otimes\mathrm{id}}
K\otimes A\otimes L
\xrightarrow{\mathrm{id}\otimes\mu}K\otimes L
\]
is to be interpreted according to~\eqref{eqcapconvention}, with no
additional Koszul sign inserted in the middle arrow. The induced twisted
differential is
\[
\partial_\phi(k\otimes l)
:=
\partial_\otimes(k\otimes l)
+(-1)^{|k|}(k\otimes l)\cap\phi.
\]
Equation~\eqref{tweq} implies $\partial_\phi^2=0$.
\end{definition}
	
	Note that the equation $\varepsilon_A\phi_1=0$ for $\phi$ is equivalent to $(\varepsilon_{K}\otimes \varepsilon_{L})\partial_\phi=0$ on the tensor product $K\otimes L$ with tensor augmentation. Consequently, $(K\otimes L, \partial_\phi)$ is an augmented $\Lambda$--complex. 
	
	\begin{definition}
	The augmented $\Lambda$--complex $(K\otimes L,\partial_\phi)$ is called a {\it twisted tensor product (over $\Lambda$)} of $K$ and $L$, and the differential $\partial_\phi$ is called the {\it twisted differential} induced by $\phi$. 
    \end{definition}
	
	In Section 9 of \cite{B}, Brown proves that one can replace $(K\otimes L,\partial_\phi)$ by another twisted tensor product $(K\otimes H_*(L),\partial_\psi)$ up to chain homotopy equivalence and recovers a theorem of Hirsch \cite{Hir} via this result. To introduce this theorem, we first give the following definition:
	\begin{definition}\label{defcalE}
Let $L$ be as above. Define the graded endomorphism algebra
\[
\begin{aligned}
\mathcal E(L)&:=\bigoplus_{q\geq0}\mathcal E(L)_q,\\
\mathcal E(L)_q&:=\mathrm{Hom}^q_\Lambda\big(H_*(L),H_*(L)\big)
=\big\{h\in\mathrm{End}_\Lambda(H_*(L))\mid h(H_p(L))\subset H_{p+q}(L)\text{ for every }p\geq0\big\}.
\end{aligned}
\]
Equipped with the trivial differential, composition as multiplication, and the identity endomorphism as unit, $\mathcal E(L)$ is a unital differential graded $\Lambda$--algebra. It acts on $H_*(L)$ by evaluation,
\[
\lambda:\mathcal E(L)\otimes H_*(L)\longrightarrow H_*(L),
\qquad h\otimes x\longmapsto h(x).
\]
\end{definition}
	
	Since all chain complexes under consideration are nonnegatively graded, $|\psi|=-1$, and Brown's construction has $\psi(K_0)=0$, only nonnegative-degree endomorphisms occur. Thus the preceding definition retains every endomorphism used in the construction.

	We use the term \emph{twisting cochain} below in the following slightly extended sense. When the target differential graded algebra is not equipped with an augmentation, a degree $-1$ map is called a twisting cochain if it vanishes on $K_0$ and satisfies the twisting-cochain identity~\eqref{tweq}. When the target is augmented, we retain the normalization condition in Definition~\ref{deftch}. This convention is needed only for Brown's homology replacement
	\[
	\psi:K\longrightarrow\mathcal E(L),
	\]
	because $\mathcal E(L)$ need not carry a canonical algebra augmentation when $H_0(L)$ has rank greater than one. No augmentation of $\mathcal E(L)$ is used in the arguments below. Brown's construction separately ensures that the induced twisted differential on $K\otimes H_*(L)$ preserves the tensor augmentation.

	\begin{theorem}[{Brown \cite[Theorem~9.1]{B}, in the present conventions]}]\label{theoremb2}
If $\Lambda$ is a principal ideal ring and $K$, $L$, and $H_*(L)$ are $\Lambda$--free, then there is a degree $-1$ morphism
\[
\psi\in\mathrm{Hom}^{-1}_\Lambda(K,\mathcal E(L))
\]
satisfying the twisting-cochain identity and inducing a twisted differential $\partial_\psi$ on $K\otimes H_*(L)$. Moreover, there is a chain homotopy equivalence
\[
\theta:(K\otimes L,\partial_\phi)\longrightarrow(K\otimes H_*(L),\partial_\psi)
\]
of augmented $\Lambda$--complexes.$\qqed$
\end{theorem}
	See Remark~\ref{rmkbr} for Brown's explicit construction of $\psi$ and $\theta$.
	
	\subsection{Brown's Twisted Tensor Product Model for Weakly Transitive Fibrations}\label{ssecbrown}
Let $C_*(Y)$ denote the singular chain complex over $\mathbb Z$ of a topological space $Y$. We call the subcomplex $C_*(Y|Z)$ generated by singular simplices which map all vertices of the standard simplex into a subspace $Z\subset Y$ the \emph{vertex--relative singular chain complex}. This is different from the ordinary relative singular chain complex $C_*(Y,Z)$: $C_*(Y|Z)$ is a subcomplex rather than a quotient, and the defining condition concerns only the vertices of each simplex. If $Z=\{z\}$, we write $C_*(Y|z)$.

Following Brown~\cite[Section~1]{B}, let $P(B)$
denote the space of Moore paths in $B$, with strictly associative
concatenation $\alpha\beta$ traversing $\alpha$ first and then $\beta$.
A lifting function for $p:E\to B$ is a continuous map
\[
\lambda:U_p\longrightarrow E,\qquad
U_p=\{(\alpha,x)\in P(B)\times E\mid \alpha(r)=p(x)\},
\]
where $r$ is the length of $\alpha$, such that
$p\lambda(\alpha,x)=\alpha(0)$.
The fibration equipped with $\lambda$ is called
\emph{weakly transitive} if $\lambda(e_{b_0},x)=x$ for
$x\in F=p^{-1}(b_0)$ and
\[
\lambda(\alpha\beta,x)=\lambda(\alpha,\lambda(\beta,x))
\]
whenever $\alpha(r)=\beta(0)=b_0$ and $(\beta,x)\in U_p$;
here $e_{b_0}$ is the length-zero constant path at $b_0$.
Consequently, the Moore loop space $\Omega B$ acts on $F$
by $\alpha\cdot x=\lambda(\alpha,x)$, inducing a
$C_*(\Omega B)$--module structure on $C_*(F)$.

\begin{theorem}[{\cite[Theorems 4.1 and 4.2]{B}}]\label{theorembr}
Let $p:E\to B$ be a weakly transitive fibration in the sense of Brown, where $(B,b_0)$ is path--connected, and set
\[
F=p^{-1}(b_0).
\]
Then there is a twisting cochain
\[
\varphi_p\in\mathrm{Hom}^{-1}_{\mathbb Z}\big(C_*(B|b_0),C_*(\Omega B)\big)
\]
which induces a twisted differential $\partial_{\varphi_p}$ on $C_*(B|b_0)\otimes C_*(F)$. It has the following properties.
\begin{itemize}
\item[(i)] If $T\in C_1(B|b_0)$ and $T_0$ is the constant $1$--simplex at $b_0$, then
\[
\varphi_p(T)=T-T_0,
\]
where $T$ and $T_0$ are viewed as $0$--simplices of the Moore loop space. The twisting cochain vanishes on constant simplices.

\item[(ii)] Brown constructs natural chain maps
\[
\chi_p:\big(C_*(B|b_0)\otimes C_*(F),\partial_{\varphi_p}\big)
\longrightarrow C_*(E|F)
\]
and
\[
\omega_p^{\mathrm{Br}}:C_*(E|F)
\longrightarrow
\big(C_*(B|b_0)\otimes C_*(F),\partial_{\varphi_p}\big)
\]
which are mutually inverse up to natural chain homotopy. Thus both maps are chain homotopy equivalences of the underlying chain complexes.

\item[(iii)] These constructions are natural with respect to morphisms of weakly transitive fibrations and their chosen lifting functions. $\qqed$
\end{itemize} 
\end{theorem}

The inverse $\omega_p^{\mathrm{Br}}$ is constructed in Brown's proof by acyclic models. We shall use this construction in Section~\ref{secsfm}, where we refine the inverse so that it is strictly compatible with the Alexander--Whitney coaction. No such strict compatibility is asserted here for Brown's forward map $\chi_p$.

\subsection{Hurewicz Replacement of Serre Fibrations}\label{ssechurrep}
We shall apply Brown's construction to Serre fibrations. Since Brown's lifting-function framework requires the covering homotopy property for arbitrary spaces, we first recall the standard mapping-path replacement.

Let
\[
p:E\longrightarrow B
\]
be a continuous map. Define
\[
E^H:=\{(e,\gamma)\in E\times B^I\mid \gamma(0)=p(e)\}
\]
and
\[
p^H:E^H\longrightarrow B,
\qquad p^H(e,\gamma)=\gamma(1).
\]
There is a canonical map
\[
j:E\longrightarrow E^H,
\qquad j(e)=\bigl(e,c_{p(e)}\bigr),
\]
where $c_{p(e)}$ is the constant path at $p(e)$. Clearly $p^Hj=p$.

\begin{proposition}\label{prophurewiczrep}
The map $p^H:E^H\to B$ is a Hurewicz fibration, and $j:E\to E^H$ is an ordinary homotopy equivalence and a map over $B$.

If $p:E\to B$ is a Serre fibration, then for every $b\in B$ the induced map
\[
j_b:p^{-1}(b)\longrightarrow(p^H)^{-1}(b)
\]
is a weak homotopy equivalence. Consequently, for every field $\Bbbk$,
\[
(j_b)_*:H_*(p^{-1}(b);\Bbbk)
\xrightarrow{\cong}
H_*((p^H)^{-1}(b);\Bbbk)
\]
is an isomorphism.
\end{proposition}

\begin{proof}
The map $p^H$ is the standard mapping-path fibration: a path beginning at $p^H(e,\gamma)=\gamma(1)$ is lifted by concatenating it to $\gamma$. Let
\[
q:E^H\longrightarrow E,
\qquad q(e,\gamma)=e.
\]
Then $qj=\mathrm{id}_E$. Shrinking $\gamma$ to its initial point gives a homotopy from $\mathrm{id}_{E^H}$ to $jq$, so $j$ is an ordinary homotopy equivalence. Notice that $j$ itself is a map over $B$, although this homotopy need not be fibrewise.

Suppose now that $p$ is a Serre fibration. Since every Hurewicz fibration is a Serre fibration, $j$ gives a morphism of Serre fibrations over $\mathrm{id}_B$. Comparing the long exact homotopy sequences of the two fibrations shows that $j_b$ induces isomorphisms on all homotopy groups, together with the corresponding statement on path components. Hence $j_b$ is a weak homotopy equivalence. Weak homotopy equivalences induce isomorphisms in singular homology, which proves the last assertion.
\end{proof}

\subsection{Brown's Transitive Replacement}\label{ssecttransrep}
We next recall the replacement from Section~1 of \cite{B}. Let $p:E\to B$ be a Hurewicz fibration in Brown's lifting-function framework. Brown constructs, over the same base space $B$, a transitive fibration
\[
p^t:E^t\longrightarrow B
\]
and maps over $B$
\[
E\underset{r}{\stackrel{s}{\rightleftarrows}}E^t
\]
whose composites are homotopic to the respective identity maps through homotopies moving points only along fibres.

\begin{proposition}[{\cite[Section~1]{B}}]\label{proptransrep}
With the preceding notation, $p^t$ is transitive, the maps $s$ and $r$ are fibre homotopy inverses over $B$, and their restrictions to every fibre are homotopy equivalences. In particular, if $F=p^{-1}(b_0)$ and $F^t=(p^t)^{-1}(b_0)$, then
\[
s_F:F\longrightarrow F^t
\]
is a homotopy equivalence. $\qqed$
\end{proposition}

Thus a Serre fibration may first be replaced by the Hurewicz fibration of Proposition~\ref{prophurewiczrep} and then by Brown's transitive replacement. We shall transport the resulting chain-level model back to the original fibration in Section~\ref{secsfm}.

\section{Dualized Twisted Tensor Products}\label{secdttp}	
Throughout this section, the ground ring is a field $\Bbbk$ of arbitrary characteristic unless otherwise stated. We fix a coaugmented $\Bbbk$--dgc
\[
(K,\partial_K,\Delta_K,\varepsilon_K,1_K),
\]
an augmented $\Bbbk$--$\partial$ga
\[
(A,\partial_A,\mu_A,\varepsilon_A),
\]
and a $\Bbbk$--dgc
\[
(L,\partial_L,\Delta_L,\varepsilon_L)
\]
which is also an augmented differential graded left $A$--module. Thus the action map
\[
\mu:A\otimes L\longrightarrow L
\]
is a morphism of augmented complexes; in particular it is compatible with the differentials and satisfies
\[
\varepsilon_L(al)=\varepsilon_A(a)\varepsilon_L(l).
\]
No compatibility between the $A$--action and the coproduct $\Delta_L$ is assumed beyond the stated hypotheses. We use multiplicative notation for the $A$--action, writing $al$ for $\mu(a\otimes l)$.

	Fixing a twisting cochain $\phi\in \mathrm{Hom}^{-1}_{\Bbbk}(K,A)$, we obtain a twisted tensor product $(K\otimes L, \partial_\phi)$ as in Section~\ref{ssecttp}. Its graded dual is the central object we study in this section: 
	\begin{definition}\label{defdual} Let $R=\mathbb{Z}$ or $\Bbbk$,
		\begin{itemize}
			\item[(i)] For a chain complex $(C,\partial)$ over $R$, let
			\[(C^\vee)^n:={\mathrm{Hom}}_R(C_n,\Bbbk)\]
			denote its $\Bbbk$--coefficient graded dualization. The dualized differential $\delta$ on $C^\vee$ is defined by
			\[
			\delta(f):=(-1)^{|f|}f\circ \partial.
			\]
			so that $(C^\vee,\delta)$ is a cochain complex over $\Bbbk$.
			\item[(ii)] For the twisted tensor product $(K\otimes L,\partial_\phi)$ over $\Bbbk$, its graded dualization $\big((K\otimes L)^\vee,\delta_\phi\big)$ is called a {\it dualized twisted tensor product (over $\Bbbk$)}. The dualized differential $\delta_\phi$ is called the {\it dualized twisted differential} induced by $\phi$.
		\end{itemize}
	\end{definition}
\begin{remark}[Signed cochain convention]\label{rmksignedcochains}
Throughout the paper, $C^*(X;\Bbbk)$ denotes the degreewise dual complex with the differential of Definition~\ref{defdual} and the product dual to the Alexander--Whitney coproduct through the canonical Koszul map. If $C^*_{\mathrm{us}}(X;\Bbbk)$ denotes the usual singular cochain dga with the usual coboundary and Alexander--Whitney cup product, then
\[
J_X:C^n(X;\Bbbk)\longrightarrow C^n_{\mathrm{us}}(X;\Bbbk),
\qquad
J_X(f)=(-1)^{n(n-1)/2}f,
\]
is a natural dga isomorphism. Thus our signed convention is equivalent to the usual singular cochain convention. In particular, when we refer to the standard $C^*(B;\Bbbk)$--module structure on $C^*(E;\Bbbk)$ induced by $p^\sharp$, it is understood in this convention, or equivalently through the natural isomorphisms $J_B$ and $J_E$.
\end{remark}

This section develops four aspects of the dualized twisted tensor product $\big((K\otimes L)^\vee,\delta_\phi\big)$. The first subsection establishes its differential graded $(K^\vee,\delta_K)$--module structure. The second recalls Brown's replacement of $L$ by $H_*(L)$. The third uses the finite-type hypothesis on $H_*(L)$ to identify $(K\otimes H_*(L))^\vee$ with $K^\vee\otimes H^*(L)$ and derives an explicit formula for $\delta_\psi$. The final subsection gives a sufficient condition for $\big(K^\vee\otimes H^*(L),\delta_\psi\big)$ to be semifree.

	\subsection{Module Structure of Dualized Twisted Tensor Products}\label{ssecmodstr}
    With the fixed data, we obtain two $\Bbbk$--dgas $(K^{\vee}, \delta_K, \Delta^\vee_K)$ and $(L^{\vee},\delta_L,\Delta^\vee_L)$, as well as the dualized twisted tensor product $\big((K\otimes L)^\vee, \delta_\phi\big)$, regarded as a cochain complex over $\Bbbk$. 
    
    Here we denote the composition $K^\vee\otimes K^\vee \hookrightarrow (K\otimes K)^\vee \stackrel{\Delta^\vee_K}\longrightarrow K^\vee$ again by $\Delta^\vee_K$ and use the same convention for $L$ and other dgcs.

	The tensor coalgebra structure on $K\otimes L$ induces a natural unital graded $\Bbbk$--algebra structure on $(K\otimes L)^\vee$. The twisted differential need not be a derivation for this full multiplication; what we need is the following differential graded $K^\vee$--module structure.

\begin{lemma}\label{modlem}
The dualized twisted tensor product
\[
\big((K\otimes L)^\vee,\delta_\phi\big)
\]
is a differential graded $(K^\vee,\delta_K)$--module. Moreover, the morphism
\[
\begin{aligned}
\iota:K^\vee&\longrightarrow(K\otimes L)^\vee,\\
b&\longmapsto b\otimes\varepsilon_L
\end{aligned}
\]
is a morphism of cochain complexes, and the $K^\vee$--module structure on $(K\otimes L)^\vee$ agrees with left multiplication through $\iota$.
\end{lemma}

\begin{proof}
Consider the $K$--coaction
\[
\rho:=\Delta_K\otimes\mathrm{id}_L:
\big(K\otimes L,\partial_\phi\big)
\longrightarrow
(K,\partial_K)\otimes\big(K\otimes L,\partial_\phi\big).
\]
We first show that $\rho$ is a chain morphism. For homogeneous $k\in K$ and $l\in L$, write
\[
\Delta_K(k)=\sum k_{(1)}\otimes k_{(2)}.
\]
By Definition~\ref{deftdiff},
\[
(k\otimes l)\cap\phi
=
\sum k_{(1)}\otimes\phi(k_{(2)})l.
\]
Coassociativity gives
\[
(\Delta_K\otimes\mathrm{id}_L)\big((k\otimes l)\cap\phi\big)
=
(\mathrm{id}_K\otimes(-\cap\phi))(\Delta_K\otimes\mathrm{id}_L)(k\otimes l).
\]
The untwisted terms in
\[
\partial_{K\otimes(K\otimes L)}\rho(k\otimes l)
\quad\text{and}\quad
\rho\partial_\phi(k\otimes l)
\]
agree because $\Delta_K$ is a chain morphism. For the twisted terms, if $\Delta_K(k)=\sum k_{(1)}\otimes k_{(2)}$, then
\[
(-1)^{|k_{(1)}|}(-1)^{|k_{(2)}|}=(-1)^{|k|},
\]
and the preceding coassociativity identity shows that the twisted terms also agree. Hence $\rho$ is a chain morphism.

By our Koszul convention, the canonical morphism
\[
K^\vee\otimes(K\otimes L)^\vee
\longrightarrow
\big(K\otimes(K\otimes L)\big)^\vee
\]
is given on homogeneous elements by
\[
(b\otimes f)(k\otimes x)=(-1)^{|f||k|}b(k)f(x).
\]
Composing this morphism with
\[
\rho^\vee:\big(K\otimes(K\otimes L)\big)^\vee\longrightarrow(K\otimes L)^\vee
\]
defines a left $K^\vee$--module structure on $(K\otimes L)^\vee$. Since both maps are morphisms of cochain complexes, for homogeneous $b\in K^\vee$ and $f\in(K\otimes L)^\vee$ we obtain
\[
\delta_\phi(b\cdot f)
=
\delta_K(b)\cdot f+(-1)^{|b|}b\cdot\delta_\phi(f).
\]
Thus $\big((K\otimes L)^\vee,\delta_\phi\big)$ is a differential graded $(K^\vee,\delta_K)$--module.

It remains to consider $\iota$. Since the $A$--action on $L$ preserves the augmentation of the underlying complex,
\[
\varepsilon_L(al)=\varepsilon_A(a)\varepsilon_L(l),
\]
and $\varepsilon_A\phi=0$. Hence
\[
(\mathrm{id}_K\otimes\varepsilon_L)\big((k\otimes l)\cap\phi\big)=0.
\]
Together with $\varepsilon_L\partial_L=0$, this gives
\[
(\mathrm{id}_K\otimes\varepsilon_L)\partial_\phi
=
\partial_K(\mathrm{id}_K\otimes\varepsilon_L).
\]
Therefore
\[
\iota:K^\vee\longrightarrow(K\otimes L)^\vee,
\qquad b\longmapsto b\otimes\varepsilon_L,
\]
is a morphism of cochain complexes. Finally, the tensor coproduct and the counit identity for $L$ show that the above $K^\vee$--action agrees with left multiplication through $\iota$:
\[
b\cdot f=\iota(b)f.
\]
This finishes the proof.
\end{proof}
	\subsection{\texorpdfstring{Replacing $L$ by $H_*(L)$}{Replacing L by H(L)}}\label{ssecrep}
	In this subsection, we present a variant of Theorem~\ref{theoremb2} to model $\big((K\otimes L)^\vee, \delta_\phi\big)$ by another dualized twisted tensor product $\big((K\otimes H_*(L))^\vee, \delta_\psi\big)$, and briefly introduce Brown's inductive construction of the twisting cochain $\psi$ and the chain homotopy equivalence $\theta$. 
	Under the additional assumption that $H_*(L)$ is of finite type, we will rewrite the dualized twisted tensor product $\big((K\otimes H_*(L))^\vee, \delta_\psi\big)$ as $\big(K^{\vee}\otimes H^*(L), \delta_\psi\big)$. 
	
	In what follows, we fix the natural isomorphism
		\[
		H^*(L)\cong H_*(L)^\vee
		\]
	given by the Universal Coefficient Theorem, and identify these two complexes.

	\begin{theorem}[{Brown \cite[Theorem~9.1]{B}, in the present conventions]}]\label{theoremb2'}
Let $(K\otimes L,\partial_\phi)$ be a twisted tensor product over $\Bbbk$, and let $\mathcal E(L)$ be as in Definition~\ref{defcalE}. Then there is a degree $-1$ morphism
\[
\psi\in\mathrm{Hom}^{-1}_\Bbbk(K,\mathcal E(L))
\]
satisfying the twisting-cochain identity and inducing a twisted differential $\partial_\psi$ on $K\otimes H_*(L)$, together with a chain homotopy equivalence
\[
\theta:(K\otimes L,\partial_\phi)\longrightarrow(K\otimes H_*(L),\partial_\psi)
\]
of augmented $\Bbbk$--complexes. $\qqed$
\end{theorem}
	
	To describe explicitly how $\delta_\psi$ is obtained from Brown's homology replacement, we recall his inductive construction of $\psi$ and $\theta$ and, when $H_*(L)$ is of finite type, derive an explicit formula for the dualized twisted differential.
	
	\begin{remark}[{Brown's explicit construction of $\psi$ and $\theta$ in Section~9 of \cite{B}}]\label{rmkbr}
		The algebraic hypotheses ensure that there exists a surjective morphism $\eta : L\rightarrow H_*(L)$ extending the canonical projection $Z_*(L)\rightarrow H_*(L)$ from cycles to homology. Moreover, $\eta$ admits a section $\xi:H_*(L)\rightarrow L$ together with a chain homotopy $h:L\rightarrow L$ satisfying
		\[
		\partial_L h + h\partial_L = \xi\eta - \mathrm{id}_L.
		\]
		
		With these data, we construct inductively a twisting cochain $\psi$ and a degree $0$ morphism $\theta':K\otimes L\rightarrow H_*(L)$.
		
		We first define $\psi$ and $\theta'$ on $K_0\otimes L$ by setting $\psi(K_0)=0$ and
		\[
		\theta'(k\otimes l):=\varepsilon_K(k)\eta(l).
		\]
		Assume that $\psi$ is already defined on $K_{\leq p}$ and that $\theta'$ is already defined on $K_{\leq p}\otimes L$. Then for any $k\in K_{p+1}$, $x\in H_*(L)$, and $l\in L$, we define
		\begin{equation}\label{eqpsi}
			\begin{aligned}
				\psi(k)x :=\;& (-1)^{p+1}\theta'\Big(\partial_Kk\otimes \xi(x)+(-1)^{p+1}\big(k\otimes \xi(x)\big)\cap \phi\Big) \\
				&-(-1)^{p+1}\nu(\Delta'_K\otimes \mathrm{id}_L)\big(k\otimes \xi(x)\big)
			\end{aligned}
		\end{equation}
		and
		\[
		\begin{aligned}
			\theta'(k\otimes l) :=\;& (-1)^{p+1}\theta'\Big(\partial_Kk\otimes h(l)+(-1)^{p+1}\big(k\otimes h(l)\big)\cap \phi\Big) \\
			&-(-1)^{p+1}\nu(\Delta'_K\otimes \mathrm{id}_L)\big(k\otimes h(l)\big),
		\end{aligned}
		\]
		where $\Delta'_K:=\Delta_K-(\mathrm{id}_K\otimes1_K)$, and $\nu$ is defined by
		\[
		\begin{aligned}
			\nu:K_{\leq p}\otimes K_{\leq p}\otimes L &\longrightarrow H_*(L),\\
			k_1\otimes k_2\otimes l &\mapsto (-1)^{|k_1|}\psi(k_1)\theta'(k_2\otimes l).
		\end{aligned}
		\]
		Note that when verifying the well-definedness of $\nu(\Delta'_K\otimes \mathrm{id}_L)$ on $K_{p+1}\otimes L$, one may need the fact that the initial condition $\psi(K_0)=0$ forces $\psi\otimes \theta'$ to vanish on $K_0\otimes K_{p+1}\otimes L$.
		Finally, define $\theta$ as the composite
		\[
		\theta : K\otimes L \xrightarrow{\Delta_K \otimes \mathrm{id}_L} K\otimes K\otimes L \xrightarrow{\mathrm{id}_K \otimes \theta'} K\otimes H_*(L).
		\]
		Brown proves that this construction is well defined; we refer to Section~9 of \cite{B} for the verification.
	\end{remark}
	
	\begin{lemma}\label{lemtheta}
The chain homotopy equivalence
\[
\theta:\big(K\otimes L,\partial_\phi\big)
\longrightarrow
\big(K\otimes H_*(L),\partial_\psi\big)
\]
is a morphism of $K$--comodules. Consequently, its dual
\[
\theta^\vee:\big(K\otimes H_*(L)\big)^\vee
\longrightarrow
(K\otimes L)^\vee
\]
is a quasi--isomorphism of differential graded $(K^\vee,\delta_K)$--modules.
\end{lemma}

\begin{proof}
Recall from Remark~\ref{rmkbr} that
\[
\theta=(\mathrm{id}_K\otimes\theta')(\Delta_K\otimes\mathrm{id}_L).
\]
The $K$--comodule structures on $K\otimes L$ and $K\otimes H_*(L)$ are induced by $\Delta_K$ on the first factor. We claim that
\[
(\Delta_K\otimes\mathrm{id}_{H_*(L)})\theta
=
(\mathrm{id}_K\otimes\theta)(\Delta_K\otimes\mathrm{id}_L).
\]
For homogeneous $k\in K$ and $l\in L$, write
\[
(\Delta_K\otimes\mathrm{id}_K)\Delta_K(k)
=
(\mathrm{id}_K\otimes\Delta_K)\Delta_K(k)
=
\sum k_{(1)}\otimes k_{(2)}\otimes k_{(3)}.
\]
Then both sides evaluated on $k\otimes l$ are
\[
\sum k_{(1)}\otimes k_{(2)}\otimes\theta'(k_{(3)}\otimes l),
\]
so the claim follows from coassociativity.

By Lemma~\ref{modlem}, the duals
\[
\big((K\otimes L)^\vee,\delta_\phi\big)
\quad\text{and}\quad
\big((K\otimes H_*(L))^\vee,\delta_\psi\big)
\]
are differential graded $(K^\vee,\delta_K)$--modules, with module structures induced by these $K$--coactions. Since $\theta$ is both a chain morphism and a $K$--comodule morphism, $\theta^\vee$ is a morphism of differential graded $(K^\vee,\delta_K)$--modules. It is a quasi--isomorphism because $\theta$ is a chain homotopy equivalence.
\end{proof}

	\subsection{\texorpdfstring{The Finite-Type Hypothesis on $H_*(L)$}{The Finite-Type Hypothesis on H(L)}}\label{ssecftpHL}
	No finite-type hypothesis is required for the differential graded $K^\vee$--module
	\[
	\big((K\otimes H_*(L))^\vee,\delta_\psi\big).
	\]
	The finite-type assumption below is used only to identify this full dual with the ordinary tensor product $K^\vee\otimes H^*(L)$. In this subsection, we therefore assume that $H_*(L)$ is of finite type and retain the fixed Universal Coefficient Theorem isomorphism $H^*(L)\cong H_*(L)^\vee$.
	\begin{lemma}\label{lemiden}
The canonical morphism
\[
\begin{aligned}
\Phi:K^\vee\otimes H^*(L)&\longrightarrow(K\otimes H_*(L))^\vee,\\
b\otimes v&\longmapsto\big(k\otimes x\mapsto(-1)^{|v||k|}b(k)v(x)\big)
\end{aligned}
\]
is an isomorphism of graded $\Bbbk$--vector spaces. Under the transported differential on the source, it is an isomorphism of cochain complexes.
\end{lemma}
\begin{proof}
For every total degree $n$,
\[
\big(K^\vee\otimes H^*(L)\big)^n
=
\bigoplus_{p+q=n}K_p^\vee\otimes H_q(L)^\vee.
\]
Since $K$ and $H_*(L)$ are nonnegatively graded, only finitely many pairs $(p,q)$ occur. Moreover, each $H_q(L)$ is finite dimensional, so the canonical Koszul map gives
\[
K_p^\vee\otimes H_q(L)^\vee
\cong
(K_p\otimes H_q(L))^\vee
\]
without any finite-dimensionality assumption on $K_p$. Taking the finite direct sum over $p+q=n$ yields the stated isomorphism in every degree. Transporting $\delta_\psi=(\partial_\psi)^\vee$ on $(K\otimes H_*(L))^\vee$ through the isomorphism $\Phi$ gives the cochain-complex statement.
\end{proof}
	When $H_*(L)$ is of finite type, we transport the dualized twisted differential to $K^\vee\otimes H^*(L)$ by
\[
\delta_\psi:=\Phi^{-1}(\partial_\psi)^\vee\Phi.
\]
Thus
\[
\Phi:\big(K^\vee\otimes H^*(L),\delta_\psi\big)
\xrightarrow{\cong}
\big((K\otimes H_*(L))^\vee,\delta_\psi\big)
\]
is an isomorphism of cochain complexes. By Lemma~\ref{modlem}, the identification is compatible with the $(K^\vee,\delta_K)$--module structures; hence $\Phi$ is an isomorphism of differential graded $(K^\vee,\delta_K)$--modules. We identify these modules from now on.

	To obtain an explicit formula for $\delta_\psi$, first consider the transpose of the endomorphism action,
\[
H^*(L)\longrightarrow\big(\mathcal E(L)\otimes H_*(L)\big)^\vee,
\qquad v\longmapsto v\circ\lambda.
\]
For homogeneous $v\in H^n(L)$, the functional $v\circ\lambda$ vanishes outside
\[
\bigoplus_{q=0}^n\mathcal E(L)_q\otimes H_{n-q}(L).
\]
Since every $H_{n-q}(L)$ is finite dimensional, this transpose factors through the ordinary tensor product
\[
\lambda^\vee:H^*(L)\longrightarrow\mathcal E(L)^\vee\otimes H^*(L).
\]
We refer to this factorized map as the $\mathcal E(L)^\vee$--coaction induced by $\lambda$. We also define a degree $1$ morphism
\[
\psi^\vee:\mathcal E(L)^\vee\longrightarrow K^\vee,
\qquad
\psi^\vee(\alpha)=(-1)^{|\alpha|}\alpha\circ\psi
\]
for homogeneous $\alpha\in\mathcal E(L)^\vee$. We now have the following explicit formula for the dualized twisted differential.
\begin{lemma}\label{lemdelta}
Let $b\in K^\vee$ and $v\in H^*(L)$ be homogeneous. Write
\[
\lambda^\vee(v)=\sum_i\alpha_i\otimes v_i,
\]
where $\alpha_i\in\mathcal E(L)^\vee$ and $v_i\in H^*(L)$ are homogeneous. Then
\[
\delta_\psi(b\otimes v)
=
\delta_K(b)\otimes v
+
\sum_i(-1)^{|b|+|\alpha_i|+1}
\Delta_K^\vee\bigl(b\otimes\psi^\vee(\alpha_i)\bigr)\otimes v_i.
\]
Here $\Delta_K^\vee$ denotes the multiplication on $K^\vee$ induced by $\Delta_K$ through the canonical Koszul morphism
\[
K^\vee\otimes K^\vee\longrightarrow(K\otimes K)^\vee.
\]
\end{lemma}

\begin{proof}
Let
\[
p=|b|,\qquad q=|v|,
\]
and set $r_i=|\alpha_i|$ and $s_i=|v_i|$, so that $r_i+s_i=q$. By the Koszul convention in Lemma~\ref{lemiden}, the identity
\[
\lambda^\vee(v)=\sum_i\alpha_i\otimes v_i
\]
means that for homogeneous $a\in\mathcal E(L)$ and $x\in H_*(L)$,
\[
v(ax)=\sum_i(-1)^{s_i|a|}\alpha_i(a)v_i(x).
\]

Let $k\in K$ and $x\in H_*(L)$ be homogeneous. Under the identification $\Phi$ of Lemma~\ref{lemiden}, the untwisted part of $\delta_\psi$ is $\delta_K\otimes\mathrm{id}_{H^*(L)}$. Hence
\[
\begin{aligned}
&\Phi\Big(\big(\delta_\psi-\delta_K\otimes\mathrm{id}_{H^*(L)}\big)(b\otimes v)\Big)(k\otimes x)\\
&\qquad=
(-1)^{p+q+|k|}\Phi(b\otimes v)\big((k\otimes x)\cap\psi\big).
\end{aligned}
\]
Writing $\Delta_K(k)=\sum_{(k)}k_{(1)}\otimes k_{(2)}$ and using
\[
(k\otimes x)\cap\psi
=
\sum_{(k)}k_{(1)}\otimes\psi(k_{(2)})x,
\]
the right-hand side becomes
\[
\sum_{i,(k)}
(-1)^{p+q+|k|+q|k_{(1)}|+s_ir_i}
 b(k_{(1)})\alpha_i\big(\psi(k_{(2)})\big)v_i(x),
\]
where a nonzero term has $|\psi(k_{(2)})|=r_i$, hence $|k_{(2)}|=r_i+1$.

On the other hand, the Koszul convention for the multiplication induced by $\Delta_K$ gives
\[
\Delta_K^\vee\bigl(b\otimes\psi^\vee(\alpha_i)\bigr)(k)
=
\sum_{(k)}(-1)^{(r_i+1)|k_{(1)}|+r_i}
 b(k_{(1)})\alpha_i\big(\psi(k_{(2)})\big).
\]
Therefore
\[
\begin{aligned}
&\Phi\Big(
(-1)^{p+r_i+1}\Delta_K^\vee\bigl(b\otimes\psi^\vee(\alpha_i)\bigr)\otimes v_i
\Big)(k\otimes x)\\
&\qquad=
\sum_{(k)}(-1)^{p+r_i+1+s_i|k|+(r_i+1)|k_{(1)}|+r_i}
 b(k_{(1)})\alpha_i\big(\psi(k_{(2)})\big)v_i(x).
\end{aligned}
\]
Using $q=r_i+s_i$ and $|k|=|k_{(1)}|+r_i+1$ on every nonzero summand, the two exponents agree modulo $2$. Summing over $i$ proves the formula.
\end{proof}

	\subsection{Semifree Dualized Twisted Tensor Products}\label{ssecsfd}
	In this subsection, we assume that $H_*(L)$ is of finite type and give a sufficient condition for $(K^{\vee}\otimes H^*(L), \delta_\psi)$ to be a semifree  $(K^{\vee},\delta_K)$--module. To this end, we need to examine more closely the difference between the twisted differential $\delta_\psi$ and the canonical tensor differential $\delta_\otimes =\delta_K\otimes \mathrm{id}_{H^*(L)}$ on $K^{\vee}\otimes H^*(L)$, which is the untwisted $(K^{\vee},\delta_K)$--module differential.
	
	\begin{lemma}\label{difflem}
For homogeneous $b\in K^\vee$ and $v\in H^*(L)$,
\[
\big(\delta_\psi-\delta_K\otimes\mathrm{id}_{H^*(L)}\big)(b\otimes v)
=
\sum_\alpha b_\alpha\otimes\big(v\circ\psi(k_\alpha)\big)
+
\sum_\beta b_\beta\otimes v_\beta,
\]
where $k_\alpha\in K_1$, $|b_\alpha|=|b|+1$, $|b_\beta|\geq|b|+2$, and $|v_\beta|\leq|v|-1$.
\end{lemma}

\begin{proof}
Set $n=|v|$. Since
\[
\lambda^\vee(v)
\in
\bigoplus_{q=0}^n\mathcal E(L)_q^\vee\otimes H^{n-q}(L),
\]
write
\[
\lambda^\vee(v)
=
\sum_{q=0}^n\sum_r\alpha_{q,r}\otimes v_{q,r},
\]
where $\alpha_{q,r}\in\mathcal E(L)_q^\vee$ and $v_{q,r}\in H^{n-q}(L)$. Lemma~\ref{lemdelta} gives
\[
\begin{aligned}
&\big(\delta_\psi-\delta_K\otimes\mathrm{id}_{H^*(L)}\big)(b\otimes v)\\
&\qquad=
\sum_{q=0}^n\sum_r(-1)^{|b|+q+1}
\Delta_K^\vee\bigl(b\otimes\psi^\vee(\alpha_{q,r})\bigr)\otimes v_{q,r}.
\end{aligned}
\]

If $q\geq1$, then $\psi^\vee(\alpha_{q,r})\in(K^\vee)^{q+1}$. Thus each corresponding term has the form $b_\beta\otimes v_\beta$ with
\[
|b_\beta|=|b|+q+1\geq|b|+2,
\qquad
|v_\beta|=n-q\leq n-1.
\]
These terms form the second sum.

For $q=0$, set
\[
\xi_v:=\sum_r\psi^\vee(\alpha_{0,r})\otimes v_{0,r}
\in K_1^\vee\otimes H^n(L).
\]
Since $H_n(L)$ is finite dimensional, so is $H^n(L)$, and the canonical map
\[
K_1^\vee\otimes H^n(L)
\xrightarrow{\cong}
\mathrm{Hom}_\Bbbk(K_1,H^n(L))
\]
is an isomorphism. Under this identification, $\xi_v$ corresponds to
\[
K_1\longrightarrow H^n(L),
\qquad
k\longmapsto v\circ\psi(k),
\]
because for $k\in K_1$ and $x\in H_n(L)$,
\[
\sum_r\psi^\vee(\alpha_{0,r})(k)v_{0,r}(x)
=
\sum_r\alpha_{0,r}(\psi(k))v_{0,r}(x)
=
v(\psi(k)x).
\]
Choose finitely many $k_\alpha\in K_1$ such that $\{v\circ\psi(k_\alpha)\}_\alpha$ is a basis of the image of this map. Then for suitable $\eta_\alpha\in K_1^\vee$,
\[
\xi_v
=
\sum_\alpha\eta_\alpha\otimes\big(v\circ\psi(k_\alpha)\big).
\]
Hence the $q=0$ contribution is
\[
\sum_\alpha
(-1)^{|b|+1}\Delta_K^\vee(b\otimes\eta_\alpha)
\otimes\big(v\circ\psi(k_\alpha)\big).
\]
Setting
\[
b_\alpha:=(-1)^{|b|+1}\Delta_K^\vee(b\otimes\eta_\alpha)
\]
gives $|b_\alpha|=|b|+1$ and yields the first sum. Combining the two parts completes the proof.
\end{proof}
	
	The preceding lemma allows us to construct a semifree filtration under the following hypothesis.
	
	\begin{definition}\label{defk1nil}
A twisting cochain
\[
\psi:K\longrightarrow\mathcal E(L)
\]
is said to satisfy {\it $K_1$--nilpotence} if for every integer $s\geq 0$ there exists a positive integer $N_s$ such that
\[
\psi(k_1)\psi(k_2)\cdots\psi(k_{N_s})x=0
\]
for all $k_1,\ldots,k_{N_s}\in K_1$ and all $x\in H_s(L)$.
\end{definition}
	
	\begin{theorem}\label{theoremsf}
Assume that $H_*(L)$ is of finite type and that the twisting cochain $\psi$ satisfies $K_1$--nilpotence. Then $(K^\vee\otimes H^*(L),\delta_\psi)$ is semifree as a $(K^\vee,\delta_K)$--module.
\end{theorem}

\begin{proof}
Fix an arbitrary integer $s\geq 0$. Set $\Gamma_{0,s}=H_s(L)$ and define inductively
\[
\Gamma_{i,s}
:=
\mathrm{span}_\Bbbk\{\psi(k)x\mid k\in K_1,\ x\in\Gamma_{i-1,s}\},
\qquad 1\leq i\leq N_s.
\]
Then
\[
H_s(L)=\Gamma_{0,s}\supset\Gamma_{1,s}\supset\cdots\supset\Gamma_{N_s,s}=0,
\]
where the last equality follows from $K_1$--nilpotence.

Via the fixed isomorphism $H^s(L)\cong H_s(L)^\vee$, define
\[
\Gamma_{i,s}^\perp
:=
\{v\in H^*(L)\mid v(\Gamma_{i,s})=0\},
\qquad 0\leq i\leq N_s.
\]
This gives the increasing filtration
\begin{equation}\label{filt1}
0=\Gamma_{0,s}^\perp
\subset\Gamma_{1,s}^\perp
\subset\cdots
\subset\Gamma_{N_s,s}^\perp=H^s(L).
\end{equation}

By Lemma~\ref{difflem},
\[
\big(K^\vee\otimes H^{\leq s}(L),\delta_\psi\big)
\]
is a $(K^\vee,\delta_K)$--submodule. Set
\[
M(s):=K^\vee\otimes H^{\leq s}(L).
\]
We obtain a filtration by $(K^\vee,\delta_K)$--submodules
\[
0=M(-1)\subset M(0)\subset\cdots\subset M(s)\subset\cdots\subset M(\infty),
\]
where
\[
M(\infty)=K^\vee\otimes H^*(L)\cong(K\otimes H_*(L))^\vee.
\]

Set
\[
Q(s):=M(s)/M(s-1)
\]
and denote its induced differential by $D_\psi$. As a graded $K^\vee$--module,
\[
Q(s)\cong K^\vee\otimes H^s(L).
\]
By Lemma~\ref{lemsemifree}, it suffices to prove that $(Q(s),D_\psi)$ is semifree over $(K^\vee,\delta_K)$.

Using~\eqref{filt1}, define
\[
F(t):=K^\vee\otimes\big(H^s(L)\cap\Gamma_{t,s}^\perp\big),
\qquad 0\leq t\leq N_s.
\]
Then
\begin{equation}\label{filt2}
0=F(0)\subset F(1)\subset\cdots\subset F(N_s)=Q(s)
\end{equation}
is a filtration by graded $K^\vee$--submodules.

By Lemma~\ref{difflem}, the induced differential on $Q(s)$ has the form
\[
D_\psi(b\otimes v)
=
\delta_K(b)\otimes v
+
\sum_\alpha b_\alpha\otimes\big(v\circ\psi(k_\alpha)\big),
\]
for homogeneous $b\in K^\vee$ and $v\in H^s(L)$, with $k_\alpha\in K_1$. Suppose $v\in\Gamma_{t,s}^\perp$. For $x\in\Gamma_{t-1,s}$,
\[
\psi(k_\alpha)x\in\Gamma_{t,s},
\]
and therefore
\[
\big(v\circ\psi(k_\alpha)\big)(x)=0.
\]
Thus
\[
v\circ\psi(k_\alpha)\in\Gamma_{t-1,s}^\perp,
\]
so
\[
\big(D_\psi-\delta_K\otimes\mathrm{id}_{H^s(L)}\big)(F(t))
\subset F(t-1)
\qquad(1\leq t\leq N_s).
\]
In particular $D_\psi(F(t))\subset F(t)$, and~\eqref{filt2} is a filtration by $(K^\vee,\delta_K)$--submodules.

On every quotient $F(t)/F(t-1)$, the induced differential is the one coming from $\delta_K$ on the $K^\vee$--factor. Since
\[
F(t)/F(t-1)
\cong
K^\vee\otimes
\frac{H^s(L)\cap\Gamma_{t,s}^\perp}{H^s(L)\cap\Gamma_{t-1,s}^\perp}
\]
as graded $K^\vee$--modules, each quotient is free over $K^\vee$ on a graded $\Bbbk$--vector space generated by cycles. Hence~\eqref{filt2} is a semifree filtration of $(Q(s),D_\psi)$.

Therefore $(Q(s),D_\psi)$ is semifree for every $s\geq0$. Applying Lemma~\ref{lemsemifree} to the filtration
\[
\{M(s)\}_{s\geq-1}
\]
proves that
\[
(K^\vee\otimes H^*(L),\delta_\psi)
\]
is semifree as a $(K^\vee,\delta_K)$--module.
\end{proof}
	
	\section{Semifree Cochain Model for a Fibration}\label{secsfm}
Throughout this section, $\Bbbk$ denotes a field of arbitrary characteristic. Integral chains are used in the geometric constructions and in Brown's original model; after tensoring with $\Bbbk$ we use the degreewise dual convention of Definition~\ref{defdual}.
For a fibre $F$, we write
\[
\mathcal E(F):=\bigoplus_{q\geq0}\mathrm{Hom}^q_\Bbbk\big(H_*(F;\Bbbk),H_*(F;\Bbbk)\big).
\]

\subsection{Vertex--Relative Singular Chain Coalgebras}\label{ssecnqi}
For a topological space $Y$ and a subspace $Z\subset Y$, recall that the vertex--relative singular chain complex $C_*(Y|Z)$ is the subcomplex generated by simplices all of whose vertices lie in $Z$. The Alexander--Whitney coproduct restricts to $C_*(Y|Z)$, so it is a sub-dgc of $C_*(Y)$. As emphasized above, this is not the ordinary relative singular chain complex $C_*(Y,Z)$.

\begin{proposition}\label{propqiso}
Let $Y$ be a topological space and let $Z\subset Y$ meet every path component of $Y$. Then there exists a morphism of $\mathbb Z$--dgcs
\[
h_Z:C_*(Y)\longrightarrow C_*(Y|Z)
\]
such that the inclusion
\[
i_Z:C_*(Y|Z)\hookrightarrow C_*(Y)
\]
and $h_Z$ are mutually inverse up to ordinary chain homotopy.

More precisely, these ordinary chain homotopies may be induced by a face-compatible geometric homotopy whose time slices
\[
H^Z_s:C_*(Y)\longrightarrow C_*(Y),
\qquad 0\leq s\leq1,
\]
are morphisms of $\mathbb Z$--dgcs and satisfy
\[
H^Z_0=\mathrm{id}_{C_*(Y)},
\qquad
H^Z_1=i_Zh_Z.
\]
Moreover, every $H^Z_s$ restricts to $C_*(Y|Z)$, and the restricted endpoints are
\[
\mathrm{id}_{C_*(Y|Z)}
\qquad\text{and}\qquad
h_Zi_Z.
\]
In particular, $i_Z$ and $h_Z$ are quasi--isomorphisms.
\end{proposition}

The construction is choice-dependent and is given explicitly in Section ~\ref{appendix} (the Appendix). We shall need choices compatible with the projection of a Serre fibration.

\begin{lemma}\label{lemfibcompatible}
Let
\[
F\xrightarrow{\hphantom{p}}E\xrightarrow{p}B
\]
be a Serre fibration, where $(B,b_0)$ is path--connected and $F=p^{-1}(b_0)$. Then the morphisms in Proposition~\ref{propqiso} may be chosen simultaneously as
\[
h_B:C_*(B)\longrightarrow C_*(B|b_0)
\]
and
\[
h_E:C_*(E)\longrightarrow C_*(E|F)
\]
so that
\begin{equation}\label{eqfibcompatible}
p_\sharp h_E=h_Bp_\sharp.
\end{equation}
Moreover, the geometric homotopies may be chosen so that every time slice $H^B_s$ and $H^E_s$ is a morphism of $\mathbb Z$--dgcs and
\begin{equation}\label{eqfibhomotopycompatible}
p_\sharp H^E_s=H^B_sp_\sharp
\qquad(0\leq s\leq1).
\end{equation}
\end{lemma}

The proof is given in the Appendix. After extension of coefficients to $\Bbbk$, set $K=C_*(B|b_0;\Bbbk)$. Equip $C_*(E;\Bbbk)$ with the left dg $K$--coaction
\[
\rho_E:=(h_Bp_\sharp\otimes\mathrm{id})AW_E,
\]
and equip $C_*(E|F;\Bbbk)$ with
\[
\rho_{E|F}:=(p_\sharp\otimes\mathrm{id})AW_E.
\]

\begin{corollary}\label{corhecomodule}
With the preceding notation,
\[
h_E:C_*(E;\Bbbk)\longrightarrow C_*(E|F;\Bbbk)
\]
is a strict morphism of left dg $K$--comodules and a chain homotopy equivalence of the underlying complexes.
\end{corollary}

\begin{proof}
Since $h_E$ is a dgc morphism and $p_\sharp h_E=h_Bp_\sharp$,
\[
\begin{aligned}
\rho_{E|F}h_E
&=(p_\sharp\otimes\mathrm{id})AW_Eh_E\\
&=(p_\sharp h_E\otimes h_E)AW_E\\
&=(h_Bp_\sharp\otimes h_E)AW_E\\
&=(\mathrm{id}_K\otimes h_E)\rho_E.
\end{aligned}
\]
The chain-homotopy-equivalence statement follows from Proposition~\ref{propqiso}.
\end{proof}

\begin{corollary}\label{corhemodule}
Set
\[
K^\vee=C^*(B|b_0;\Bbbk).
\]
Then the degreewise dual
\[
h_E^\vee:C^*(E|F;\Bbbk)\longrightarrow C^*(E;\Bbbk)
\]
is a quasi--isomorphism of differential graded $K^\vee$--modules. Here the $K^\vee$--module structure on $C^*(E;\Bbbk)$ is induced by
\[
K^\vee\xrightarrow{\;h_B^\vee\;}C^*(B;\Bbbk)\xrightarrow{\;p^\sharp\;}C^*(E;\Bbbk),
\]
whereas the $K^\vee$--module structure on $C^*(E|F;\Bbbk)$ is induced by the dual of
\[
p_\sharp:C_*(E|F;\Bbbk)\longrightarrow K.
\]
In fact, $h_E^\vee$ is a cochain homotopy equivalence of the underlying cochain complexes. $\qqed$
\end{corollary}

\subsection{A Strict Comodule Refinement of Brown's Inverse}\label{ssecbrowncomodule}
Let $F\xrightarrow{\hphantom{p}}E\xrightarrow{p}B$ be a weakly transitive fibration in the sense of Brown, and set
\[
K=C_*(B|b_0),\qquad L=C_*(F),\qquad C=C_*(E|F),\qquad
T_\varphi=(K\otimes L,\partial_\varphi).
\]
Equip $C$ with $\rho_E:=(p_\sharp\otimes\mathrm{id}_C)AW_E$ and $T_\varphi$ with $\rho_\varphi:=\Delta_K\otimes\mathrm{id}_L$. The same coassociativity calculation as in Lemma~\ref{modlem}, valid over $\mathbb Z$, shows that $\rho_\varphi$ is a chain map.

\medskip
\noindent\textit{A cofree-comodule observation.}\ 
Let $D$ be a counital left $K$--comodule with coaction $\rho_D$. A degree-zero map $f:D\to K\otimes L$ is $K$--colinear if and only if
\begin{equation}\label{eqcofreeformula}
f=(\mathrm{id}_K\otimes\bar f)\rho_D,
\qquad
\bar f=(\varepsilon_K\otimes\mathrm{id}_L)f.
\end{equation}
Indeed, the forward implication follows by applying $\mathrm{id}_K\otimes\varepsilon_K\otimes\mathrm{id}_L$ to the colinearity identity, and the converse follows from coassociativity. In particular, for a singular $n$--simplex $\sigma\in C_n(E|F)$,
\begin{equation}\label{eqcofreeomega}
f(\sigma)=\sum_{i=0}^n p_\sharp\sigma[0,\ldots,i]\otimes\bar f\bigl(\sigma[i,\ldots,n]\bigr).
\end{equation}
Since all vertices of $\sigma$ lie in $F$, the $i=0$ term is $b_0\otimes\bar f(\sigma)$; hence, once $f$ is fixed below dimension $n$, the only undetermined part of a colinear value on $\sigma$ is its fibre component $\bar f(\sigma)$.

Suppose now that $D$ and $K\otimes L$ are dg $K$--comodules, with differentials $\partial_D$ and $\partial_T$, and that $f$ is colinear. The degree $-1$ map
\[
\mathcal D_f:=\partial_Tf-f\partial_D
\]
is again colinear in the graded sense:
\begin{equation}\label{eqdefectcolinear}
\rho_T\mathcal D_f=(\mathrm{id}_K\otimes\mathcal D_f)\rho_D.
\end{equation}
To see the sign explicitly, for homogeneous $k\in K$ and $x\in D$ the difference
\[
\partial_{K\otimes T}(\mathrm{id}_K\otimes f)(k\otimes x)
-(\mathrm{id}_K\otimes f)\partial_{K\otimes D}(k\otimes x)
\]
is $(-1)^{|k|}k\otimes(\partial_Tf-f\partial_D)(x)$, which is exactly $(\mathrm{id}_K\otimes\mathcal D_f)(k\otimes x)$ by the Koszul convention.

\medskip
\noindent\textit{Brown's universal model.}\ 
We also recall the part of Brown's acyclic-model construction used below. Let $q_n:\Delta^n\to B_n$ identify all vertices of $\Delta^n$ to a single point $b_n$, and let $E(B_n)$ be the Moore path space. If $e(a)$ denotes the terminal point of a Moore path $a$, Brown defines
\[
X_n=\big\{(a,t)\in E(B_n)\times\Delta^n\mid e(a)=q_n(t)\big\},
\qquad p_n(a,t)=a(0),
\]
and equips $p_n:X_n\to B_n$ with the transitive lifting function induced by path concatenation. Write $\mathcal F_n=(X_n,B_n,p_n,\lambda_n)$ and $F_n=p_n^{-1}(b_n)$. The distinguished singular $n$--simplex is
\[
\zeta_n(t)=\big(c_{q_n(t)},t\big),
\]
where $c_{q_n(t)}$ is the constant Moore path; hence $\zeta_n\in C_n(X_n|F_n)$. For every $\sigma\in C_n(E|F)$, Brown constructs from $\sigma$ and the chosen lifting function a morphism of weakly transitive fibrations $(u_\sigma,v_\sigma):\mathcal F_n\to(E,B,p)$ satisfying $(u_\sigma)_\sharp(\zeta_n)=\sigma$; see~\cite[(6.1)(i)]{B}.

Set $K_{(n)}=C_*(B_n|b_n)$ and $L_{(n)}=C_*(F_n)$, let $\varphi_n:K_{(n)}\to C_*(\Omega B_n)$ be Brown's twisting cochain for $\mathcal F_n$, and write $T_{\varphi_n}=(K_{(n)}\otimes L_{(n)},\partial_{\varphi_n})$. Brown's construction also gives
\begin{equation}\label{equniversalacyclic}
H_q(F_n;\mathbb Z)=0\quad(q>0),
\qquad
H_q(T_{\varphi_n})=0\quad(q>0);
\end{equation}
see~\cite[(6.1)(ii)--(6.2)]{B}. The first vanishing follows from the comparison of $F_n$ with the loop space of $B_n$, while $B_n$ has the homotopy type of a wedge of circles.

\begin{lemma}\label{lembrowncomodule}
There exists a chain homotopy equivalence of the underlying chain complexes
\[
\omega_p:C\longrightarrow T_\varphi
\]
which is a strict morphism of left dg $K$--comodules; equivalently,
\[
\rho_\varphi\omega_p=(\mathrm{id}_K\otimes\omega_p)\rho_E.
\]
\end{lemma}

\begin{proof}
We refine Brown's inverse by induction on simplex dimension. In dimensions $0$ and $1$ we retain Brown's original choices; their formulas have the form~\eqref{eqcofreeomega}, hence are colinear.

Assume that the map has been constructed below dimension $n\ge2$ and is there both colinear and a chain map. On the universal simplex $\zeta_n$, first set the still-undetermined fibre component equal to zero, $\bar{\widetilde\omega}_n(\zeta_n)=0$, and let $\widetilde\omega_n$ be the corresponding colinear extension given by~\eqref{eqcofreeformula}. Consider the degree $-1$ map $\mathcal D_n:=\partial_{\varphi_n}\widetilde\omega_n-\widetilde\omega_n\partial$ and set $D_n:=\mathcal D_n(\zeta_n)$. By~\eqref{eqdefectcolinear} and the Alexander--Whitney formula,
\[
\rho_{\varphi_n}(D_n)
=\sum_{i=0}^n(-1)^i(p_n)_\sharp\zeta_n[0,\ldots,i]\otimes\mathcal D_n\bigl(\zeta_n[i,\ldots,n]\bigr).
\]
Every term with $i>0$ vanishes by the induction hypothesis, so $\rho_{\varphi_n}(D_n)=b_n\otimes D_n$. Applying $\mathrm{id}_{K_{(n)}}\otimes\varepsilon_{K_{(n)}}\otimes\mathrm{id}_{L_{(n)}}$ gives $D_n=b_n\otimes z_n$, where $z_n:=(\varepsilon_{K_{(n)}}\otimes\mathrm{id}_{L_{(n)}})(D_n)\in C_{n-1}(F_n)$.

The element $D_n$ is a cycle: since $\widetilde\omega_n$ agrees with the previously constructed map on $\partial\zeta_n$,
\[
\partial_{\varphi_n}D_n
=-\partial_{\varphi_n}\omega_p(\partial\zeta_n)
=-\omega_p(\partial^2\zeta_n)=0.
\]
Because $\varphi_n$ vanishes on the distinguished vertex, $\partial_{\varphi_n}(b_n\otimes z_n)=b_n\otimes\partial z_n$, hence $\partial z_n=0$. By~\eqref{equniversalacyclic}, choose $y_n\in C_n(F_n)$ with $\partial y_n=-z_n$. Replacing the fibre component by $\bar\omega_n(\zeta_n)=y_n$, i.e.
\[
\omega_n(\zeta_n)=\widetilde\omega_n(\zeta_n)+b_n\otimes y_n,
\]
preserves colinearity and kills $D_n$. Transporting this fixed choice along Brown's model maps $(u_\sigma,v_\sigma)$ defines the corrected value on every singular $n$--simplex $\sigma$. Induction therefore produces a strict dg $K$--comodule map $\omega_p:C\to T_\varphi$.

Finally, compare $\omega_p$ with Brown's original inverse $\omega_p^{\mathrm{Br}}$. If an ordinary chain homotopy $H$ has been constructed below dimension $n$, then
\[
c_n=(\omega_n-\omega_n^{\mathrm{Br}})(\zeta_n)-H(\partial\zeta_n)
\]
is a cycle. The second vanishing in~\eqref{equniversalacyclic} provides $w_n$ with $\partial_{\varphi_n}w_n=c_n$; setting $H(\zeta_n)=w_n$ and transporting this choice gives the next stage. Thus $\omega_p\simeq\omega_p^{\mathrm{Br}}$ as ordinary chain maps. Since $\omega_p^{\mathrm{Br}}$ is a chain homotopy inverse of Brown's forward map, $\omega_p$ is a chain homotopy equivalence. No comodule compatibility of this comparison homotopy is needed.
\end{proof}

\begin{corollary}\label{corbrowncomoduledual}
Let $\Bbbk$ be a field and retain the hypotheses of Lemma~\ref{lembrowncomodule}. After tensoring with $\Bbbk$, the morphism
\[
\omega_p:C_*(E|F;\Bbbk)
\longrightarrow
\big(C_*(B|b_0;\Bbbk)\otimes C_*(F;\Bbbk),\partial_\varphi\big)
\]
is a strict morphism of left dg $C_*(B|b_0;\Bbbk)$--comodules and a chain homotopy equivalence. Consequently, its degreewise dual
\[
\omega_p^\vee:
\Big(\big(C_*(B|b_0;\Bbbk)\otimes C_*(F;\Bbbk)\big)^\vee,\delta_\varphi\Big)
\longrightarrow
C^*(E|F;\Bbbk)
\]
is a quasi--isomorphism of differential graded $C^*(B|b_0;\Bbbk)$--modules.
\end{corollary}

\begin{proof}
Tensor the chain-map, comodule, and chain-homotopy identities with $\Bbbk$ and then dualize. No finite-type hypothesis is required.
\end{proof}

\subsection{The Based Cochain Model}\label{ssecbasedmodel}
We now combine the algebraic results of Section~\ref{secdttp} with the preceding chain-level constructions. Model existence and semifreeness are logically separate: the based differential graded module model exists without degreewise nilpotence, whereas degreewise nilpotence will be used only to prove semifreeness.

\begin{theorem}\label{theorembasedwt}
Let
\[
F\xrightarrow{\hphantom{p}}E\xrightarrow{p}B
\]
be a weakly transitive fibration with $(B,b_0)$ path--connected. Set $K=C_*(B|b_0;\Bbbk)$ and $L=C_*(F;\Bbbk)$; then $K^\vee=C^*(B|b_0;\Bbbk)$. There exists a twisting cochain
\[
\psi:K\longrightarrow\mathcal E(L)
\]
and a quasi--isomorphism of differential graded $K^\vee$--modules
\[
\zeta_p:\Big((K\otimes H_*(F;\Bbbk))^\vee,\delta_\psi\Big)
\longrightarrow C^*(E;\Bbbk).
\]
More precisely,
\[
\zeta_p=h_E^\vee\circ\omega_p^\vee\circ\theta^\vee.
\]
If $H_*(F;\Bbbk)$ is of finite type, the source is canonically identified with
\[
\big(K^\vee\otimes H^*(F;\Bbbk),\delta_\psi\big).
\]
\end{theorem}

\begin{proof}
Brown's twisting cochain gives the twisted tensor product $(K\otimes L,\partial_\varphi)$. By Theorem~\ref{theoremb2'}, there is a twisting cochain $\psi$ and a chain homotopy equivalence
\[
\theta:(K\otimes L,\partial_\varphi)
\longrightarrow
(K\otimes H_*(L),\partial_\psi).
\]
By Lemma~\ref{lemtheta}, $\theta$ is a strict morphism of left dg $K$--comodules. Hence
\[
\theta^\vee:\Big((K\otimes H_*(L))^\vee,\delta_\psi\Big)
\longrightarrow
\Big((K\otimes L)^\vee,\delta_\varphi\Big)
\]
is a quasi--isomorphism of differential graded $K^\vee$--modules. Corollary~\ref{corbrowncomoduledual} gives the strict $K^\vee$--linear quasi--isomorphism $\omega_p^\vee$, and Corollary~\ref{corhemodule} gives the strict $K^\vee$--linear quasi--isomorphism $h_E^\vee$. Their composite is $\zeta_p$. The finite-type identification is Lemma~\ref{lemiden}.
\end{proof}

The following transfer principle isolates the amount of fibrewise information needed to transport a based model.

\begin{proposition}\label{propfibertransfer}
Let
\[
F\xrightarrow{\hphantom{p}}E\xrightarrow{p}B,
\qquad
F'\xrightarrow{\hphantom{p'}}E'\xrightarrow{p'}B
\]
be Serre fibrations over the same path--connected based space $(B,b_0)$. Suppose that
\[
s:E\longrightarrow E'
\]
is a map over $B$, $p's=p$, and an ordinary homotopy equivalence. Suppose moreover that its restriction to the fibres induces an isomorphism
\[
\alpha:=(s_F)_*:H_*(F;\Bbbk)\xrightarrow{\cong}H_*(F';\Bbbk).
\]
Set $K=C_*(B|b_0;\Bbbk)$, so that $K^\vee=C^*(B|b_0;\Bbbk)$. If a twisting cochain
\[
\psi':K\longrightarrow\mathcal E(F')
\]
gives a quasi--isomorphism of differential graded $K^\vee$--modules
\[
\zeta':\Big((K\otimes H_*(F';\Bbbk))^\vee,\delta_{\psi'}\Big)
\longrightarrow C^*(E';\Bbbk),
\]
then there is a twisting cochain
\[
\psi:K\longrightarrow\mathcal E(F)
\]
and a quasi--isomorphism of differential graded $K^\vee$--modules
\[
\zeta:\Big((K\otimes H_*(F;\Bbbk))^\vee,\delta_\psi\Big)
\longrightarrow C^*(E;\Bbbk).
\]
\end{proposition}

\begin{proof}
Define
\[
\psi(k):=\alpha^{-1}\psi'(k)\alpha.
\]
Conjugation by $\alpha$ preserves composition in the graded endomorphism algebra and hence preserves the twisting-cochain identity. Since $\alpha$ is induced by a map of fibres, it preserves the homology augmentation, $\varepsilon_{F'}\alpha=\varepsilon_F$. Furthermore,
\[
(\varepsilon_K\otimes\varepsilon_{F'})(\mathrm{id}_K\otimes\alpha)\partial_\psi
=(\varepsilon_K\otimes\varepsilon_{F'})\partial_{\psi'}(\mathrm{id}_K\otimes\alpha)=0,
\]
so $(\varepsilon_K\otimes\varepsilon_F)\partial_\psi=0$. Thus the tensor-augmentation compatibility is preserved. Moreover,
\[
\mathrm{id}_K\otimes\alpha:(K\otimes H_*(F;\Bbbk),\partial_\psi)
\longrightarrow
(K\otimes H_*(F';\Bbbk),\partial_{\psi'})
\]
is an isomorphism of twisted chain complexes and a strict isomorphism of left dg $K$--comodules. Therefore
\[
\widehat\alpha:=\big((\mathrm{id}_K\otimes\alpha)^\vee\big)^{-1}
\]
is an isomorphism of differential graded $K^\vee$--modules in the opposite direction on the duals.

Since $p's=p$, we have $s^\sharp(p')^\sharp=p^\sharp$. Thus $s^\sharp:C^*(E';\Bbbk)\to C^*(E;\Bbbk)$ is a strict $K^\vee$--module morphism, where both module structures are induced through $h_B^\vee$. It is a quasi--isomorphism because $s$ is a homotopy equivalence. Hence
\[
\zeta:=s^\sharp\circ\zeta'\circ\widehat\alpha
\]
has the required properties.
\end{proof}

\begin{theorem}\label{theorembasedmodel}
Let
\[
F\xrightarrow{\hphantom{p}}E\xrightarrow{p}B
\]
be a Serre fibration, where $(B,b_0)$ is path--connected. Set $K=C_*(B|b_0;\Bbbk)$; thus $K^\vee=C^*(B|b_0;\Bbbk)$.
Then there exists a twisting cochain
\[
\psi:K\longrightarrow\mathcal E(F)
\]
and a quasi--isomorphism of differential graded $K^\vee$--modules
\[
\zeta:\Big((K\otimes H_*(F;\Bbbk))^\vee,\delta_\psi\Big)
\xrightarrow{\simeq}
C^*(E;\Bbbk).
\]
If $H_*(F;\Bbbk)$ is of finite type, the source is canonically identified with
\[
\big(K^\vee\otimes H^*(F;\Bbbk),\delta_\psi\big).
\]
\end{theorem}

\begin{proof}
First apply Proposition~\ref{prophurewiczrep} to obtain the mapping-path Hurewicz replacement
\[
p^H:E^H\longrightarrow B
\]
and the canonical map $j:E\to E^H$. The map $j$ is an ordinary homotopy equivalence over $\mathrm{id}_B$, and its restriction to the fibres induces an isomorphism
\[
(j_F)_*:H_*(F;\Bbbk)\xrightarrow{\cong}H_*(F^H;\Bbbk),
\qquad F^H=(p^H)^{-1}(b_0).
\]
Since $p^H$ is a Hurewicz fibration, Brown's transitive replacement applies. Let
\[
p^t:E^t\longrightarrow B
\]
be the resulting transitive fibration and $s:E^H\to E^t$ Brown's fibre homotopy equivalence over $B$. Its restriction $s_F:F^H\to F^t$ is a homotopy equivalence.

Consequently,
\[
\sigma:=s\circ j:E\longrightarrow E^t
\]
is an ordinary homotopy equivalence over $B$, while its restriction to the fibres induces an isomorphism
\[
\alpha:=\sigma_{F*}:H_*(F;\Bbbk)\xrightarrow{\cong}H_*(F^t;\Bbbk).
\]
The transitive fibration is weakly transitive, so Theorem~\ref{theorembasedwt} gives a twisting cochain $\psi^t$ and a based model for $E^t$. Proposition~\ref{propfibertransfer}, applied to $\sigma$, transports this model to $E$ and gives
\[
\psi(k)=\alpha^{-1}\psi^t(k)\alpha.
\]
The finite-type assertion follows from Lemma~\ref{lemiden}; no nilpotence condition has been used.
\end{proof}

Since conjugation by $\alpha$ preserves the twisting-cochain identity and the evaluation action, all algebraic results of Section~\ref{secdttp}, in particular Theorem~\ref{theoremsf}, apply to the transported twisting cochain $\psi$.

\subsection{Degreewise Nilpotence and Semifreeness}\label{ssecnil}
For the remainder of this section, set $K=C_*(B|b_0;\Bbbk)$, so that $K^\vee=C^*(B|b_0;\Bbbk)$, and retain the transported twisting cochain $\psi$ of Theorem~\ref{theorembasedmodel}.

\begin{definition}\label{defnilact}
Let $G$ be a group and let $M$ be a $\Bbbk$--vector space equipped with a linear left $G$--action. Define $M_0:=M$ and, for $i\geq1$,
\[
M_i:=\operatorname{span}_{\Bbbk}\{gm-m\mid g\in G,\ m\in M_{i-1}\}.
\]
which gives a descending filtration
\[
M=M_0\supset M_1\supset M_2\supset\cdots.
\]
The action is called {\it nilpotent} if $M_N=0$ for some $N>0$. The least such $N$ is called the nilpotence index of this action and is denoted by
\[
\mathrm{nil}(G,M).
\]
For a graded $\Bbbk$--vector space $V=\{V^n\}$ with a linear left $G$--action presevring degrees, the action is said to be {\it degreewise nilpotent} if $G$ acts nilpotently on each $V^n$.
\end{definition}

\begin{lemma}\label{lemnil}
Let $\psi$ be the twisting cochain in Theorem~\ref{theorembasedmodel}. Then for every singular $1$--simplex
\[
T\in C_1(B|b_0;\Bbbk)
\]
and every $x\in H_*(F;\Bbbk)$,
\begin{equation}\label{eqpsi1top}
\psi(T)x=([T]-1)x,
\end{equation}
where $[T]\in\pi_1(B,b_0)$ and the right-hand side is the usual action on the homology of the fibre. Consequently, $\psi$ satisfies $K_1$--nilpotence in Definition \ref{defk1nil} if and only if the action of $\pi_1(B)$ on $H_*(F;\Bbbk)$ is degreewise nilpotent.
\end{lemma}

\begin{proof}
First consider the transitive replacement $p^t:E^t\to B$ and its twisting cochain $\psi^t$. For a singular $1$--simplex $T\in C_1(B|b_0;\Bbbk)$, Brown's inductive construction of $\psi^t$ gives
\[
\psi^t(T)y=([T]-1)y
\qquad(y\in H_*(F^t;\Bbbk)).
\]
Indeed, in degree one Brown's formula for the homology replacement reduces to the action of $\varphi(T)$, and Theorem~\ref{theorembr} gives $\varphi(T)=T-T_0$.

Let
\[
\alpha:H_*(F;\Bbbk)\xrightarrow{\cong}H_*(F^t;\Bbbk)
\]
be the fibre homology isomorphism used in Theorem~\ref{theorembasedmodel}. The naturality of fundamental group actions with respect to maps of fibrations over $\mathrm{id}_B$ makes $\alpha$ $\pi_1(B)$--equivariant. Since $\psi(T)=\alpha^{-1}\psi^t(T)\alpha$,
\[
\psi(T)x=\alpha^{-1}\big(([T]-1)\alpha(x)\big)=([T]-1)x.
\]

Suppose first that the $\pi_1(B)$--action is degreewise nilpotent and set
\[
N_s=\mathrm{nil}\big(\pi_1(B),H_s(F;\Bbbk)\big).
\]
For $k=\sum_ja_jT_j\in K_1$ and $x\in H_s(F;\Bbbk)$, Equation~\eqref{eqpsi1top} gives
\[
\psi(k)x=\sum_ja_j([T_j]-1)x,
\]
so $\psi(k)M_{i-1}\subset M_i$. Therefore
\[
\psi(k_1)\cdots\psi(k_{N_s})x=0
\]
for all $k_i\in K_1$ and all $x\in H_s(F;\Bbbk)$, proving $K_1$--nilpotence.

Conversely, suppose $K_1$--nilpotence holds with exponent $N_s$ for $H_s(F;\Bbbk)$. Given $g_1,\ldots,g_{N_s}\in\pi_1(B)$, choose based singular $1$--simplices $T_i$ representing them. Then
\[
(g_1-1)\cdots(g_{N_s}-1)x
=\psi(T_1)\cdots\psi(T_{N_s})x=0.
\]
It follows inductively from Definition~\ref{defnilact} that $M_{N_s}=0$. Hence the action is degreewise nilpotent.
\end{proof}

\begin{theorem}\label{coretheorem1}
Assume that $H_*(F;\Bbbk)$ is of finite type. If the fundamental group $\pi_1(B)$ acts degreewise nilpotently on $H_*(F;\Bbbk)$, then
\[
\big(K^\vee\otimes H^*(F;\Bbbk),\delta_\psi\big)
\]
is a semifree differential graded $K^\vee$--module.
\end{theorem}

\begin{proof}
By Lemma~\ref{lemnil}, the degreewise nilpotence of the $\pi_1(B)$--action is equivalent to the $K_1$--nilpotence condition of Definition~\ref{defk1nil}. The conclusion follows from Theorem~\ref{theoremsf}.
\end{proof}

\begin{theorem}\label{coretheorem2}
Assume that $H_*(F;\Bbbk)$ is of finite type and that the fundamental group $\pi_1(B)$ acts degreewise nilpotently on $H_*(F;\Bbbk)$. Then the quasi--isomorphism
\[
\zeta:\big(K^\vee\otimes H^*(F;\Bbbk),\delta_\psi\big)
\longrightarrow C^*(E;\Bbbk)
\]
constructed in Theorem~\ref{theorembasedmodel} is a semifree $K^\vee$--module resolution of $C^*(E;\Bbbk)$.
\end{theorem}

\begin{proof}
Theorem~\ref{theorembasedmodel} gives the $K^\vee$--linear quasi--isomorphism $\zeta$ without any nilpotence assumption, and Theorem~\ref{coretheorem1} makes its source semifree under the present hypothesis.
\end{proof}

\subsection{\texorpdfstring{Base Change to $C^*(B;\Bbbk)$}{Base Change to Ordinary Cochains}}\label{ssecbc}
Set $K=C_*(B|b_0;\Bbbk)$ and $R=C^*(B;\Bbbk)$, and let $\alpha=h_B^\vee:K^\vee\to R$.
By Proposition~\ref{propqiso}, $\alpha$ is a quasi--isomorphism of $\Bbbk$--dgas. Under the hypotheses of Theorem~\ref{coretheorem2}, write $M_0=\big(K^\vee\otimes H^*(F;\Bbbk),\delta_\psi\big)$ and form $M:=R\otimes_{K^\vee}M_0$. As a graded $R$--module, $M\cong R\otimes H^*(F;\Bbbk)$, and we continue to denote the induced differential by $\delta_\psi$.

\begin{proposition}\label{theorembc}
With the preceding notation:
\begin{itemize}
\item[(i)] The differential graded $R$--module
\[
M=\big(R\otimes H^*(F;\Bbbk),\delta_\psi\big)
\]
is semifree.

\item[(ii)] The canonical $\alpha$--linear morphism
\[
\beta:M_0\longrightarrow M,
\qquad x\longmapsto1\otimes_{K^\vee}x,
\]
is a quasi--isomorphism.

\item[(iii)] There is a unique morphism of differential graded $R$--modules
\[
m:M\longrightarrow C^*(E;\Bbbk)
\]
such that $m\beta=\zeta$, where $C^*(E;\Bbbk)$ is equipped with its standard $R$--module structure induced by $p^\sharp$. Explicitly,
\begin{equation}\label{eqbasem}
m(r\otimes_{K^\vee}x)=p^\sharp(r)\,\zeta(x).
\end{equation}
Moreover, $m$ is a quasi--isomorphism.
\end{itemize}
\end{proposition}

\begin{proof}
Part~(i) follows from Lemma~\ref{lembc}, since $M_0$ is semifree over $K^\vee$. Part~(ii) follows from Lemma~\ref{lemqibc}, because $\alpha=h_B^\vee$ is a quasi--isomorphism of dgas.

For~(iii), $\zeta$ is $K^\vee$--linear with respect to the $K^\vee$--module structure on $C^*(E;\Bbbk)$ induced by $p^\sharp\alpha$. Hence the extension--restriction of scalars adjunction gives a unique $R$--linear map $m$, necessarily given by Equation~\eqref{eqbasem}. It is well defined because, for $b\in K^\vee$,
\[
\begin{aligned}
m(r\alpha(b)\otimes x)
&=p^\sharp(r\alpha(b))\zeta(x)\\
&=p^\sharp(r)p^\sharp(\alpha(b))\zeta(x)\\
&=p^\sharp(r)\zeta(bx)\\
&=m(r\otimes bx).
\end{aligned}
\]
It is a chain map and $R$--linear by construction. Finally, $m\beta=\zeta$ strictly. Since both $\beta$ and $\zeta$ are quasi--isomorphisms, the two--out--of--three property implies that $m$ is a quasi--isomorphism.
\end{proof}

\subsection{Proofs and Comparison Results}\label{ssecpf12}
\noindent{\it Proof of Theorem~\ref{mthm1}.}\ \ 
Part~(i) is Theorem~\ref{theorembasedmodel}. Under the finite-type hypothesis, Lemma~\ref{lemiden} identifies its source with $C^*(B|b_0;\Bbbk)\otimes H^*(F;\Bbbk)$.

For~(ii), Theorem~\ref{coretheorem1} shows that this based resolution is semifree when the action is degreewise nilpotent. Proposition~\ref{theorembc} then performs base change along $h_B^\vee:C^*(B|b_0;\Bbbk)\to C^*(B;\Bbbk)$ and gives the stated semifree resolution over ordinary cochains. $\qqed$

\begin{corollary}\label{mcor}
Consider a morphism of Serre fibrations
\[
\begin{tikzcd}
F \arrow[r] \arrow[d,"h"] & E \arrow[r,"p"] \arrow[d,"f"] & B \arrow[d,"g"]\\
F' \arrow[r] & E' \arrow[r,"p'"] & B',
\end{tikzcd}
\]
where both fibrations satisfy the hypotheses of Theorem~\ref{mthm1}(ii). Let $m$ and $m'$ be the corresponding semifree cochain models. Then there exists a $g^\sharp$--linear morphism of differential graded modules
\[
\widetilde f:
\big(C^*(B';\Bbbk)\otimes H^*(F';\Bbbk),\delta_{\psi'}\big)
\longrightarrow
\big(C^*(B;\Bbbk)\otimes H^*(F;\Bbbk),\delta_\psi\big)
\]
such that
\[
m\widetilde f\simeq_{C^*(B';\Bbbk)}f^\sharp m'.
\]
The homotopy class of $\widetilde f$ as a $C^*(B';\Bbbk)$--module morphism is unique, where the bottom model is regarded as a $C^*(B';\Bbbk)$--module by restriction of scalars along $g^\sharp$.
\end{corollary}

\begin{proof}
Set $R'=C^*(B';\Bbbk)$ and $R=C^*(B;\Bbbk)$, and regard every $R$--module below as an $R'$--module by restriction along $g^\sharp$. Since $p'f=gp$,
\[
f^\sharp(p')^\sharp=p^\sharp g^\sharp,
\]
so $f^\sharp:C^*(E';\Bbbk)\to C^*(E;\Bbbk)$ is a strict $R'$--module morphism. Hence $f^\sharp m'$ is an $R'$--module morphism. The source of $m'$ is semifree over $R'$, while $m$ is a quasi--isomorphism of $R'$--modules. Proposition~\ref{proplift} gives $\widetilde f$ with
\[
m\widetilde f\simeq_{R'}f^\sharp m'.
\]
Its $R'$--module homotopy class is unique by the same proposition.
\end{proof}

\begin{remark}\label{rmkmodelnaturality}
The morphism $\widetilde f$ in Corollary~\ref{mcor} is obtained from the comparison property of semifree resolutions. The corollary therefore does not assert that the choices entering the twisting cochains or the replacement constructions are strictly functorial. Rather, a morphism of fibrations determines a comparison morphism between the resulting semifree models, unique up to module homotopy.
\end{remark}

\noindent{\it Proof of Theorem~\ref{mthm2}.}\ \ 
Set $R=C^*(B;\Bbbk)$. By Theorem~\ref{mthm1},
\[
m:\big(R\otimes H^*(F;\Bbbk),\delta_\psi\big)\longrightarrow C^*(E;\Bbbk)
\]
is a semifree $R$--module resolution, while $m_W:(R\otimes W,d)\to C^*(E;\Bbbk)$ is another one. Proposition~\ref{proplift} gives an $R$--module morphism
\[
n:(R\otimes W,d)\longrightarrow\big(R\otimes H^*(F;\Bbbk),\delta_\psi\big)
\]
such that $mn\simeq_Rm_W$. Applying the uniqueness statement of Proposition~\ref{proplift} in both directions shows that $n$ is a homotopy equivalence of differential graded $R$--modules and hence a quasi--isomorphism.

Regard $\Bbbk$ as a right differential graded $R$--module through the augmentation $\varepsilon_B$. Since $n$ is an $R$--module homotopy equivalence, choose an $R$--linear homotopy inverse $n'$ and $R$--linear homotopies between $n'n$, $nn'$ and the corresponding identity maps. Tensoring these maps and homotopies with $\Bbbk$ over $R$ shows directly that
\[
\bar n:=\Bbbk\otimes_Rn:
\Bbbk\otimes_R(R\otimes W)
\longrightarrow
\Bbbk\otimes_R\big(R\otimes H^*(F;\Bbbk),\delta_\psi\big)
\]
is a chain homotopy equivalence. The first complex is $(W,\bar d)$. The second is canonically $(H^*(F;\Bbbk),0)$: every twisting term has positive $R$--degree and is annihilated by the augmentation. Under these identifications, $\bar n$ is the morphism in the statement. $\qqed$

\section{Connections to Spectral Sequences}\label{secss}
Throughout this section, the ground ring is a field $\Bbbk$ of arbitrary characteristic unless otherwise stated. We first recall the spectral sequence associated with a semifree filtration.

\subsection{\texorpdfstring{Spectral Sequence of a Semifree $(R,d)$--Module}{Spectral Sequence of a Semifree Module}}
We follow the convention of Chapter~18 of \cite{FHT}. For a differential $(R,d)$--module $(M,d)$ with cohomological differential, set $M_{-k}=M^k$ in order to use the homological indexing associated with an increasing filtration.

Let $(M,d)$ be equipped with an increasing filtration
\[
\mathcal F:\cdots\subset F_p\subset F_{p+1}\subset\cdots
\]
by differential graded $R$--submodules. For $r\geq0$, define
\[
Z^r_p:=\{x\in F_p\mid dx\in F_{p-r}\},
\qquad
D^r_p:=F_p\cap dF_{p+r}.
\]
Define
\[
E^0_{p,q}:=(F_p/F_{p-1})_{p+q}
\]
and, for $r\geq1$,
\[
E^r_{p,q}:=\big(Z^r_p/(Z^{r-1}_{p-1}+D^{r-1}_p)\big)_{p+q}.
\]
The differential induces
\[
d^r:E^r_{p,q}\longrightarrow E^r_{p-r,q+r-1},
\]
and
\[
E^{r+1}\cong H_*(E^r,d^r).
\]
\begin{definition}
The spectral sequence $(E^r,d^r)$ is called the {\it homology spectral sequence} of the increasing filtration $\mathcal F$.
\end{definition}

We now specialize to a Serre fibration satisfying the hypotheses of Theorem~\ref{mthm1}(ii). Set $K=C_*(B|b_0;\Bbbk)$ and $R=C^*(B;\Bbbk)$, so that $K^\vee=C^*(B|b_0;\Bbbk)$.
Recall that the based model
\[
M_0=\big(K^\vee\otimes H^*(F;\Bbbk),\delta_\psi\big)
\]
is equipped, by the proof of Theorem~\ref{theoremsf}, with a semifree filtration
\[
0=M_0(-1)\subset M_0(0)\subset M_0(1)\subset\cdots.
\]
The semifree model of Theorem~\ref{mthm1}(ii) is obtained from $M_0$ by extension of scalars along
\[
h_B^\vee:K^\vee\longrightarrow R:
\]
\[
M=R\otimes_{K^\vee}M_0\cong\big(R\otimes H^*(F;\Bbbk),\delta_\psi\big).
\]
Accordingly, $M(k):=R\otimes_{K^\vee}M_0(k)$ is a semifree filtration of $M$.

\noindent{\it Proof of Theorem~\ref{mthm3}.}\ \ 
Set
\[
N=\mathrm{nil}\big(\pi_1(B),H_*(F;\Bbbk)\big).
\]
Define $\Lambda_0=H_*(F;\Bbbk)$ and, inductively,
\[
\Lambda_i=\operatorname{span}_{\Bbbk}\{gx-x\mid g\in\pi_1(B),\ x\in\Lambda_{i-1}\},
\qquad1\leq i\leq N.
\]
Then
\[
H_*(F;\Bbbk)=\Lambda_0\supset\Lambda_1\supset\cdots\supset\Lambda_N=0.
\]
Via $H^*(F;\Bbbk)\cong H_*(F;\Bbbk)^\vee$, set
\[
\Gamma_i:=\Lambda_i^\perp.
\]
Thus
\[
0=\Gamma_0\subset\Gamma_1\subset\cdots\subset\Gamma_N=H^*(F;\Bbbk).
\]
For each $j\geq0$, set
\[
\Gamma_{i,j}=\Gamma_i\cap H^j(F;\Bbbk).
\]
By Lemma~\ref{lemnil}, Equation~\eqref{eqpsi1top}, and the proof of Theorem~\ref{theoremsf},
\[
\delta_\psi(1\otimes\Gamma_{i,j})
\subset
K^\vee\otimes\big(H^{<j}(F;\Bbbk)\oplus\Gamma_{i-1,j}\big).
\]
The terms preserving the fibre degree lower the nilpotence filtration, whereas the remaining terms lower the fibre degree, as in Lemma~\ref{difflem}.

For $1\leq i\leq N$ and $0\leq j\leq D-1$, define
\[
\widetilde V(i,j):=H^{<j}(F;\Bbbk)\oplus\Gamma_{i,j}.
\]
The sum is direct because $\Gamma_{i,j}\subset H^j(F;\Bbbk)$. These subspaces form an increasing filtration
\[
\begin{aligned}
0&\subset\widetilde V(1,0)\subset\cdots\subset\widetilde V(N,0)\\
&\subset\widetilde V(1,1)\subset\cdots\subset\widetilde V(N,1)\\
&\subset\cdots\subset\widetilde V(N,D-1)=H^*(F;\Bbbk).
\end{aligned}
\]
After reindexing it as
\[
0=V(-1)\subset V(0)\subset\cdots\subset V(ND-1),
\]
the modules $K^\vee\otimes V(k)$ form a semifree filtration of $M_0$. Base change along $h_B^\vee$ gives
\[
0=R\otimes V(-1)\subset\cdots\subset R\otimes V(ND-1)=R\otimes H^*(F;\Bbbk),
\]
a semifree filtration of the model in Theorem~\ref{mthm1}(ii).

Its nonzero filtration degrees lie in $0\leq p\leq ND-1$. Since $d^r$ lowers filtration degree by $r$,
\[
d^r=0\qquad(r\geq ND).
\]
Thus $E^{ND}=E^\infty$, and the spectral sequence collapses at the $E^{ND}$--page. $\qqed$

\subsection{Relation with the Eilenberg--Moore Spectral Sequence}
The degreewise nilpotence hypothesis in Theorem~\ref{mthm1}(ii) is classical in Eilenberg--Moore theory. Dwyer proved that, for a Serre fibration over a connected base and an arbitrary abelian coefficient group $A$, the mod--$A$ Eilenberg--Moore spectral sequence converges strongly to the homology of the fibre if and only if the $\pi_1(B)$--action on $H_*(F;A)$ degreewise nilpotent; see~\cite{Dwyer}. Under our hypotheses, the semifree model gives a direct description of fibre cohomology as differential graded Tor.

For comparison with the standard algebraic construction, let $(C,d)$ be an $(R,d)$--module with a semifree resolution
\[
m:(M,d)=(R\otimes W,d)\longrightarrow(C,d),
\]
where
\[
W=\bigoplus_{k\geq0}W(k),
\qquad
M(k)=R\otimes\bigoplus_{i\leq k}W(i),
\qquad
dW(k)\subset M(k-1).
\]
The associated filtration has
\[
E^0_{k,*}=R\otimes W(k),
\qquad
E^1_{k,*}=H^*(R)\otimes W(k),
\]
and, together with $m(0)$, gives the sequence
\[
\cdots\xrightarrow{d^1}H^*(R)\otimes W(1)
\xrightarrow{d^1}H^*(R)\otimes W(0)
\xrightarrow{H^*(m(0))}H^*(C).
\]
Following Chapter~20 of~\cite{FHT}, such a resolution is called an \emph{Eilenberg--Moore resolution} when this sequence is exact, equivalently when $(H^*(R)\otimes W,d^1)$ is a free resolution of the $H^*(R)$--module $H^*(C)$.

\begin{corollary}\label{cortorfiber}
Let $p:E\to B$ satisfy the hypotheses of Theorem~\ref{mthm1}(ii), and set
\[
R=C^*(B;\Bbbk),
\]
and regard $\Bbbk$ as a right differential graded $R$--module through the augmentation determined by $b_0$. Then
\[
\operatorname{Tor}^{R}\bigl(\Bbbk,C^*(E;\Bbbk)\bigr)
\cong
H^*(F;\Bbbk)
\]
as graded $\Bbbk$--vector spaces, where differential graded Tor is computed using a semifree $R$--module resolution of $C^*(E;\Bbbk)$.
\end{corollary}

\begin{proof}
Let
\[
m_W:(R\otimes W,d)\longrightarrow C^*(E;\Bbbk)
\]
be any semifree $R$--module resolution. By Theorem~\ref{mthm2}, reduction along the augmentation gives a quasi--isomorphism
\[
\Bbbk\otimes_R(R\otimes W,d)
\xrightarrow{\simeq}
\bigl(H^*(F;\Bbbk),0\bigr).
\]
The cohomology of the complex on the left computes
$\operatorname{Tor}^{R}(\Bbbk,C^*(E;\Bbbk))$, which proves the assertion.
\end{proof}

If $m_W$ is an Eilenberg--Moore resolution, the reduced filtration on
$\Bbbk\otimes_R(R\otimes W,d)$ has the familiar algebraic second page
\[
E^2_{k,*}
\cong
\operatorname{Tor}^{H^*(B;\Bbbk)}_k
\bigl(\Bbbk,H^*(E;\Bbbk)\bigr).
\]
Corollary~\ref{cortorfiber} identifies the differential graded Tor target with the cohomology of the fibre. We do not use this observation as an independent convergence theorem. Dwyer's theorem supplies the classical strong-convergence criterion in the homological setting. Our construction relates the degreewise nilpotence of the action to the semifreeness of the dualized Brown model.

\section{Appendix}\label{appendix}
In this appendix, we give the geometric details underlying Proposition~\ref{propqiso} and Lemma~\ref{lemfibcompatible}. To make the construction easier to verify, we separate it into four steps: the collapsing maps $CL^n$, the induced vertex--relative dgc morphism, the face-compatible homotopy, and the compatibility with the projection of a Serre fibration.

\begin{lemma}\label{lemappcl}
For every $n\geq0$ there is a continuous map
\[
CL^n:\Delta^n\longrightarrow J^n,
\]
where $J^n$ is obtained from $\Delta^n$ by attaching one copy of $[0,1]$ at each vertex. The maps $CL^n$ send the $k$-th vertex to the outer endpoint of the attached interval, collapse the boundary of the $1/3$--corner at that vertex to the attaching point, and are compatible with all face inclusions.
\end{lemma}

\begin{proof}
For $0\leq k\leq n$, let
\[
\Delta_k^n
:=
\left\{\sum_{i=0}^n a_i e_i\in\Delta^n\ \middle|\ a_k\geq 2/3\right\}
\]
be the $1/3$--corner at $e_k$, and let
\[
F_k:=\left\{\sum_{i=0}^n a_i e_i\in\Delta^n\ \middle|\ a_k=2/3\right\}
\]
be its face opposite $e_k$. The complement of the interiors of the corners is
\[
P_n:=\left\{\sum_{i=0}^n a_i e_i\in\Delta^n\ \middle|\ 0\leq a_i\leq2/3\text{ for all }i\right\}.
\]
Let
\[
f_n:\coprod_{k=0}^n\{0\}_k\longrightarrow\Delta^n,
\qquad 0_k\longmapsto e_k,
\]
and set
\[
J^n:=\Delta^n\cup_{f_n}\Big(\coprod_{k=0}^n[0,1]_k\Big).
\]

On $\Delta_k^n$, define
\[
CL^n\Big(\sum_{i=0}^n a_i e_i\Big)=(3a_k-2)_k\in[0,1]_k.
\]
Thus $CL^n(e_k)=1_k$ and $CL^n(F_k)=\{0_k\}=\{[e_k]\}$.

On the central region $P_n$, for $x=\sum_i a_i e_i$, set
\[
w_k(x):=a_k\prod_{j\neq k}(2/3-a_j),
\qquad
W(x):=\sum_{r=0}^n w_r(x),
\qquad
\lambda_k(x):=\frac{w_k(x)}{W(x)},
\]
and define
\[
CL^n(x):=\sum_{k=0}^n\lambda_k(x)[e_k]
\in\langle[e_0],\ldots,[e_n]\rangle\subset J^n.
\]
The denominator $W(x)$ is strictly positive on $P_n$: if some $a_j=2/3$, then $w_j(x)>0$, while if all $a_j<2/3$, at least one $a_k$ is positive and then $w_k(x)>0$. Hence the $\lambda_k$ are well defined and continuous. On $F_k$ one has $\lambda_k=1$ and $\lambda_j=0$ for $j\neq k$, so the two definitions of $CL^n$ agree. Therefore $CL^n$ is continuous.

It remains to verify compatibility with faces. Let
\[
\delta^r:\Delta^{n-1}\longrightarrow\Delta^n
\]
be the standard inclusion of the $r$-th face, and let
\[
J(\delta^r):J^{n-1}\longrightarrow J^n
\]
be the induced inclusion, which sends the interval at the $i$-th vertex to the interval at the corresponding vertex of the $r$-th face. On every corner the equality
\[
CL^n\circ\delta^r=J(\delta^r)\circ CL^{n-1}
\]
is immediate from the formula $(3a_k-2)_k$. On the central region, $a_r=0$, so $w_r=0$, while every remaining weight acquires the same additional factor $2/3$. After normalization this common factor cancels. Thus the same equality holds on $P_{n-1}$, proving face compatibility.
\end{proof}

\begin{lemma}\label{lemapphz}
Let $Y$ be a topological space and let $Z\subset Y$ meet every path component. For every $y\in Y$, choose a point $r(y)\in Z$ in the same path component and a path
\[
\gamma_y:[0,1]\longrightarrow Y,
\qquad
\gamma_y(0)=y,
\qquad
\gamma_y(1)=r(y),
\]
with $r(y)=y$ and $\gamma_y$ constant when $y\in Z$. Then the maps $CL^n$ of Lemma~\ref{lemappcl} induce a morphism of $\mathbb Z$--dgcs
\[
h_Z:C_*(Y)\longrightarrow C_*(Y|Z).
\]
\end{lemma}

\begin{proof}
For a singular $n$--simplex $\sigma:\Delta^n\to Y$, define
\[
\Phi^Z_\sigma:J^n\longrightarrow Y
\]
by $\Phi^Z_\sigma|_{\Delta^n}=\sigma$ and
\[
\Phi^Z_\sigma|_{[0,1]_k}=\gamma_{\sigma(e_k)}.
\]
The definitions agree at the attaching points. Set
\[
h_Z(\sigma):=\Phi^Z_\sigma\circ CL^n.
\]
Since $CL^n(e_k)=1_k$, all vertices of $h_Z(\sigma)$ lie in $Z$, so $h_Z(\sigma)\in C_n(Y|Z)$.

Let $\alpha:[m]\hookrightarrow[n]$ be an order-preserving injection and let $\delta_\alpha:\Delta^m\to\Delta^n$ be the corresponding face inclusion. Lemma~\ref{lemappcl} and the definition of $\Phi^Z_\sigma$ give
\[
CL^n\circ\delta_\alpha=J(\delta_\alpha)\circ CL^m,
\qquad
\Phi^Z_\sigma\circ J(\delta_\alpha)=\Phi^Z_{\sigma\circ\delta_\alpha}.
\]
Consequently,
\[
h_Z(\sigma)\circ\delta_\alpha=h_Z(\sigma\circ\delta_\alpha).
\]
Taking codimension-one faces gives $\partial h_Z=h_Z\partial$. Applying the same identity to the front and back faces in the Alexander--Whitney formula gives
\[
AW\,h_Z=(h_Z\otimes h_Z)AW.
\]
The counit is preserved as well. Hence $h_Z$ is a morphism of $\mathbb Z$--dgcs.
\end{proof}

\begin{lemma}\label{lemapphomotopy}
With the choices of Lemma~\ref{lemapphz}, there is a face-compatible homotopy
\[
\mathcal K^n:\Delta^n\times[0,1]\longrightarrow J^n
\]
from the standard inclusion of $\Delta^n$ into the simplex part of $J^n$ to $CL^n$. The induced maps
\[
H^Z_s(\sigma):=\Phi^Z_\sigma\circ\mathcal K^n_s
\qquad(0\leq s\leq1)
\]
are morphisms of $\mathbb Z$--dgcs and satisfy
\[
H^Z_0=\mathrm{id}_{C_*(Y)},
\qquad
H^Z_1=i_Zh_Z.
\]
Moreover, if $\sigma\in C_n(Y|Z)$, then $H^Z_s(\sigma)\in C_n(Y|Z)$ for every $s$.
\end{lemma}

\begin{proof}
Let
\[
r^n:J^n\longrightarrow\Delta^n
\]
be the retraction which is the identity on the simplex part and collapses every attached interval to its attaching point, and set
\[
g^n:=r^n\circ CL^n.
\]
By Lemma~\ref{lemappcl}, the maps $g^n$ are compatible with all face inclusions. Since $\Delta^n$ is convex,
\[
A^n_s(x):=(1-s)x+s\,g^n(x)
\qquad(0\leq s\leq1)
\]
is a face-compatible homotopy from $\mathrm{id}_{\Delta^n}$ to $g^n$.

There is also a face-compatible homotopy from $g^n$ to $CL^n$: it is constant on the central region, and on the $k$-th corner it is given by
\[
B^n_s\Big(\sum_i a_i e_i\Big)=\bigl(s(3a_k-2)\bigr)_k.
\]
The formulas agree along $F_k$, where $3a_k-2=0$. Define
\[
\mathcal K^n_s=
\begin{cases}
A^n_{2s},&0\leq s\leq\frac12,\\
B^n_{2s-1},&\frac12\leq s\leq1.
\end{cases}
\]
This is the required face-compatible homotopy.

For every face inclusion $\delta_\alpha$, face compatibility gives
\[
H^Z_s(\sigma)\circ\delta_\alpha=H^Z_s(\sigma\circ\delta_\alpha).
\]
Hence every $H^Z_s$ is a chain map, and the same front/back-face argument as in Lemma~\ref{lemapphz} gives
\[
AW\,H^Z_s=(H^Z_s\otimes H^Z_s)AW.
\]
The counit is preserved, so each time slice is a dgc morphism. The endpoint identities follow from the definition of $\mathcal K^n$. If $\sigma\in C_n(Y|Z)$, then every path $\gamma_{\sigma(e_k)}$ is constant, and therefore every time slice remains in $C_n(Y|Z)$.
\end{proof}

\begin{proof}[Proof of Proposition~\ref{propqiso}]
Choose $r(y)$ and $\gamma_y$ as in Lemma~\ref{lemapphz}; no continuity of the assignment $y\mapsto r(y)$ or of the family $\{\gamma_y\}$ is required. Lemma~\ref{lemapphz} gives the dgc morphism
\[
h_Z:C_*(Y)\longrightarrow C_*(Y|Z).
\]
By Lemma~\ref{lemapphomotopy}, the face-compatible geometric homotopy from $\mathrm{id}$ to $i_Zh_Z$ yields, through the usual prism construction, an ordinary chain homotopy
\[
i_Zh_Z\simeq\mathrm{id}_{C_*(Y)}.
\]
The restriction of the same homotopy to $C_*(Y|Z)$ yields
\[
h_Zi_Z\simeq\mathrm{id}_{C_*(Y|Z)}.
\]
Thus $i_Z$ and $h_Z$ are mutually inverse up to ordinary chain homotopy, and the more precise assertions about the time slices follow from Lemma~\ref{lemapphomotopy}.
\end{proof}

\begin{lemma}\label{lemappfibcompat}
Let
\[
F\xrightarrow{\hphantom{p}}E\xrightarrow{p}B
\]
be a Serre fibration over the path-connected based space $(B,b_0)$. Choose, for every $b\in B$, a path $\gamma_b$ from $b$ to $b_0$, with $\gamma_{b_0}$ constant. For every $e\in E$, choose a lift $\widetilde\gamma_e$ of $\gamma_{p(e)}$ beginning at $e$, taking the constant lift when $e\in F$. Let $h_B,h_E$ and $H^B_s,H^E_s$ be the maps obtained from Lemmas~\ref{lemapphz} and~\ref{lemapphomotopy}. Then
\[
p_\sharp h_E=h_Bp_\sharp,
\qquad
p_\sharp H^E_s=H^B_sp_\sharp
\quad(0\leq s\leq1).
\]
\end{lemma}

\begin{proof}
For a singular simplex $\sigma:\Delta^n\to E$, let $\Phi^E_\sigma:J^n\to E$ and $\Phi^B_{p\sigma}:J^n\to B$ be the corresponding maps. On the simplex part,
\[
p\Phi^E_\sigma=p\sigma=\Phi^B_{p\sigma},
\]
and on the $k$-th attached interval,
\[
p\bigl(\Phi^E_\sigma|_{[0,1]_k}\bigr)
=p\widetilde\gamma_{\sigma(e_k)}
=\gamma_{p(\sigma(e_k))}
=\Phi^B_{p\sigma}|_{[0,1]_k}.
\]
Hence $p\circ\Phi^E_\sigma=\Phi^B_{p\sigma}$ on $J^n$. Composing with $CL^n$ gives $p_\sharp h_E=h_Bp_\sharp$, and composing with every time slice $\mathcal K^n_s$ gives $p_\sharp H^E_s=H^B_sp_\sharp$.
\end{proof}

\begin{proof}[Proof of Lemma~\ref{lemfibcompatible}]
Choose the paths $\gamma_b$ and lifts $\widetilde\gamma_e$ as in Lemma~\ref{lemappfibcompat}. Lemmas~\ref{lemapphz} and~\ref{lemapphomotopy} show that all the resulting maps and time slices are morphisms of $\mathbb Z$--dgcs, while Lemma~\ref{lemappfibcompat} gives the two required compatibility identities.
\end{proof}


\end{document}